\documentclass[12pt]{amsart}
\pdfoutput=1
\usepackage[margin=1in]{geometry}
\usepackage{amsmath,amsthm,amssymb,amsrefs,bbm,color,enumitem,esint,esvect,float,graphicx,mathrsfs}
\usepackage{euscript,latexsym}
\usepackage{accents, xcolor}
\usepackage{comment}

\usepackage{epstopdf}
\AppendGraphicsExtensions{.tif}

\usepackage{tikz}
\usetikzlibrary{arrows}

\usepackage{bbm}

\usepackage{hyperref}
\hypersetup{colorlinks}

\newcommand{\R}{\mathbb{R}}

\newcommand{\N}{\mathbb{N}}
\newcommand{\Z}{\mathbb{Z}}

\newcommand{\avfI}{\langle f\rangle_{I}}

\newcommand{\fhaarI}{\langle f,h_I\rangle }
\newcommand{\fhaarIsib}{\langle f,h_{I^s}\rangle }
\newcommand{\fhaarJ}{\langle f,h_J\rangle }

\newcommand{\fonehaarJminus}{\langle f_1,h_{J_{-}}\rangle }
\newcommand{\fonehaarJplus}{\langle f_1,h_{J_{+}}\rangle }
\newcommand{\ftwohaarJminus}{\langle f_2,h_{J_{-}}\rangle }
\newcommand{\ftwohaarJplus}{\langle f_2,h_{J_{+}}\rangle }

\newcommand{\bonejhaarJplus}{\langle b_{1,j},h_{J_{+}}\rangle }
\newcommand{\btwokhaarJminus}{\langle b_{2,k},h_{J_{-}}\rangle }
\newcommand{\bonehaarJplus}{\langle b_{1},h_{J_{+}}\rangle }

\newcommand{\btwohaarJminus}{\langle b_{2},h_{J_{-}}\rangle }

\newcommand{\bonejhaarIjhat}{\langle b_{1,j},h_{\hat{I}_{j}}\rangle }
\newcommand{\bonejhaarIj}{\langle b_{1,j},h_{{I}_{j}}\rangle }
\newcommand{\btwokhaarIkhat}{\langle b_{2,k},h_{\hat{I}_{k}}\rangle }
\newcommand{\btwokhaarIk}{\langle b_{2,k},h_{{I}_{k}}\rangle }

\newcommand{\gonehaarJplus}{\langle g_{1},h_{J_{+}}\rangle }

\newcommand{\gtwohaarJminus}{\langle g_{2},h_{J_{-}}\rangle }

\newcommand{\avfoneI}{\langle f_1\rangle_{I}}
\newcommand{\avftwoI}{\langle f_2\rangle_{I}}
\newcommand{\avfoneS}{\langle f_1\rangle_{S}}
\newcommand{\avftwoT}{\langle f_2\rangle_{T}}

\newcommand{\avftwoJ}{\langle f_2\rangle_{J}}

\newcommand{\avbI}{\langle b\rangle_{I}}
\newcommand{\avbIsib}{\langle b\rangle_{I^s}}
\newcommand{\avbIhat}{\langle b\rangle_{\hat{I}}}
\newcommand{\avbIplus}{\langle b\rangle_{I_{+}}}
\newcommand{\avbIminus}{\langle b\rangle_{I_{-}}}
\newcommand{\bhaarI}{\langle b,h_I\rangle }

\newcommand{\BMO}{\mathrm{BMO}}

\newcommand{\cexp}{\mathsf{E}}

\newcommand{\D}{\mathscr{D}}
\newcommand{\Hd}{\mathcal{H}}

\newcommand{\Maximal}{\mathcal{M}}
\newcommand{\MaximalD}{\Maximal_\D}

\newcommand{\Ahat}{\widehat{A}}

\newcommand{\musib}{[\mu]_{\mathrm{sib}}}
\newcommand{\Hilbert}{\mathbb{H}}

\newtheorem{thm}{Theorem}[section]
\newtheorem{lemma}[thm]{Lemma}
\newtheorem{cor}[thm]{Corollary}
\newtheorem{prop}[thm]{Proposition}
\newtheorem{rem}[thm]{Remark}

\numberwithin{equation}{section}

\newtheorem{ltheorem}{Theorem}

\def\one{\mbox{1\hspace{-4.25pt}\fontsize{12}{14.4}\selectfont\textrm{1}}}

\theoremstyle{definition}

\begin{document}

\title[Commutator Estimates for dyadic operators]{Commutator estimates for Haar shifts with general measures}
\author[T. Borges]{Tainara Borges}
\address{Tainara Borges \hfill\break\indent 
 Department of Mathematics \hfill\break\indent 
 Brown University \hfill\break\indent 
 151 Thayer Street\hfill\break\indent 
 Providence, RI, 02912 USA}
\email{tainara\_gobetti\_borges@brown.edu}
\author[J.M. Conde-Alonso]{Jos\'{e} M. Conde Alonso}
\thanks{J. M. Conde-Alonso was supported by grants CNS2022-135431 and RYC2019-027910-I (Ministerio de Ciencia, Spain)}
\address{Jos\'{e} M. Conde Alonso \hfill\break\indent 
 Departamento de Matem\'aticas \hfill\break\indent 
 Universidad Aut\'onoma de Madrid \hfill\break\indent 
 C/ Francisco Tom\'as y Valiente sn\hfill\break\indent 
 28049 Madrid, Spain}
\email{jose.conde@uam.es}
\author[J. Pipher]{Jill Pipher}
\address{Jill Pipher \hfill\break\indent 
 Department of Mathematics \hfill\break\indent 
 Brown University \hfill\break\indent 
 151 Thayer Street\hfill\break\indent 
 Providence, RI, 02912 USA}
\email{jill\_pipher@brown.edu}
\author[N.A. Wagner]{Nathan A. Wagner}
\thanks{N. A. Wagner was supported in part by National Science Foundation grant DMS \# 2203272.}

\address{Nathan A. Wagner \hfill\break\indent 
 Department of Mathematics \hfill\break\indent 
Brown University \hfill\break\indent 151 Thayer St \hfill\break\indent 
Providence, RI, 02912 USA}
\email{nathan\_wagner@brown.edu}

\date{\today}

\subjclass[2010]{}%


\begin{abstract} 
We study $L^p(\mu)$ estimates for the commutator $[\Hd,b]$, where the operator $\Hd$ is a dyadic model of the classical Hilbert transform introduced in \cite{DKPS23,domelevopetermichl} and is adapted to a non-doubling Borel measure $\mu$ satisfying a dyadic regularity condition which is necessary for $\Hd$ to be bounded on $L^p(\mu)$. We show that $\|[\Hd, b]\|_{L^p(\mu) \rightarrow L^p(\mu)} \lesssim \|b\|_{\BMO(\mu)}$, but to {\it characterize} martingale BMO requires additional commutator information.  We prove weighted inequalities for $[\Hd, b]$  
together with a version of the John-Nirenberg inequality adapted to appropriate weight classes $\Ahat_p$ that
we define for our non-homogeneous setting. This requires establishing reverse H\"{o}lder inequalities for these new weight classes. Finally, we revisit the appropriate class of nonhomogeneous measures $\mu$ for the study of different types of Haar shift operators. 
\end{abstract}

\maketitle

\section*{Introduction}

In this paper, we consider the relationship between the martingale $\BMO$ space and commutators of a dyadic Hilbert transform that was recently introduced in \cite{domelevopetermichl} and \cite{DKPS23}, but studied here in the context of non-doubling measures.

Motivated by showing a factorization theorem for the Hardy space, whose dual had just been identified as $\BMO$, 
the seminal paper \cite{CoifmanRochbergWeiss} considered the commutators of Calder\'on-Zygmund singular integral operators 
($K$) with the operator of multiplication by
a $\BMO$ function, $b$. With $b$ also denoting that multiplication operator, they showed that these commutators, 
$[K,b](\cdot):= Kb(\cdot) -bK(\cdot)$, are bounded maps of $L^p(\mathbb R^n, dx)$ to itself, with operator norm a constant
times $\|b\|_{BMO}$. Conversely, they proved that if all the commutators of Riesz transforms, $[R_j, \phi]$, with multiplication by a 
function $\phi$ are bounded maps of $L^p$, for some $1<p<\infty$, then $\phi \in \text{BMO}$, thus {\it characterizing} $\BMO$ in 
terms of an operator norm. Specifically, for $\mathbb R$, this amounts to saying that the boundedness of the commutator with the Hilbert transform 
$\Hilbert f (x) = \mathrm{p.v} \int_\R \frac{f(y)}{x-y} dy$
characterizes $\text{BMO}(\R)$.

Subsequent studies took this foundational result in a number of different and expanded directions. In one direction, \cite{CLMS} linked
such commutator estimates to their div-curl approach to compensated compactness. In the theory of martingales, which is the context for this paper, 
commutators of martingale transforms were investigated by \cite{Jan}, for {\it regular} martingales, and later by \cite{Treil13} in the ``nonregular" or nonhomogeneous setting. Endpoint estimates were investigated in \cite{Bonami2023}. 

The characterization of $\BMO$ in terms of commutators of Calder\'on-Zygmund operators extends to certain measure spaces beyond $\R^n$, equipped with the Lebesgue measure. In particular, commutators of nondegenerate Calder\'on-Zygmund operators characterize $\BMO(\mathbb{X},\mu)$ for certain metric measure spaces where the measure satisfies the {\it doubling condition} (defined below) whereas we are interested in the non-doubling setting.

Absent the doubling condition, the definition of appropriate function spaces that are connected to the boundedness of 
singular integral operators, such as the Hardy spaces
or $\BMO$ spaces, can be a difficult matter. One example of success in this context is the $\mathrm{RBMO}(\mu)$ of \cite{RBMOTolsa} for 
measures on 
$\R^n$ with polynomial growth.

Our focus here is on commutators of certain discrete, or martingale, operators that, in the doubling setting,
have been good models for Calder\'on-Zygmund operators. We work with Borel measures $\mu$ on $\R$ defined on any dyadic system $\D$ - in particular, the usual one. The system $\D = \cup_{k\in\Z} \D_k$ is a two-sided sequence of partitions that generates a martingale filtration on $\R$. This equips the space with a natural martingale $\BMO$ space \cite{Garsia}, that will be denoted $\BMO(\mu)$. See
Section \ref{sec1} for the definition and properties. 
The dyadic operator we consider here is a modern version of Petermichl's Haar shift, termed 
``the dyadic Hilbert transform" in \cite{domelevopetermichl}. 
The classical dyadic Haar shift operator of \cite{Petermichl2000} has 
been studied in the non-doubling setting in \cite{LSMP}, and more recently in \cite{CAPW}. In \cite{LSMP}, it was shown that
the boundedness of these classical Haar shift 
operators on $L^p$ spaces, for $p \neq 2$, necessitated additional constraints on the non-doubling measure. 

We now introduce (as in \cite{LSMP}) the Haar wavelet basis on $\mathbb R(d\mu)$, together with a quantity that normalizes the Haar functions in $L^2(d\mu)$.

Intervals $I \in \mathscr{D}$ have parents $\widehat{I}$, children $I_-$ and $I_+$, and a sibling $I^s$. Our measures $\mu$ fail to satisfy
the {\it doubling} doubling property, i.e., $\mu(\widehat{I}) \lesssim \mu(I)$ with an implied constant independent of $I$. In this setting, the following
quantity arises naturally in order to define Haar functions and, following \cite{LSMP}, we denote:
$$
m(I):=\frac{\mu(I_{-})\mu(I_{+})}{\mu(I)}.
$$
A measure $\mu$ on $\R$ is said to be {\it balanced} if $m(I) \sim m(\widehat{I})$ for all $I \in \D$. The class of balanced measures arises naturally in the study of Haar shifts - \cite{LSMP} showed that this property characterizes the $L^p$-boundedness of the classical
Haar shift operator of \cite{Petermichl2000} and thus the entire class of {\it higher complexity} Haar shifts (defined in Subsection \ref{subsec14}). 

Haar shift operators are linear maps that alter the Haar expansion of functions in terms of the Haar basis elements:
\begin{equation*}
    h_{I}(x)=\sqrt{m(I)}\left(\frac{\mathbbm{1}_{I_{-}}}{\mu(I_{-})}-\frac{\mathbbm{1}_{I_{+}}}{\mu(I_{+})}\right), \; I \in \D.
\end{equation*}
Above and throughout, we use standard notation for integral averages and pairings:  
\begin{equation*}
    \langle f\rangle_{I} =\frac{1}{\mu(I)}\int_{I} fd\mu, \quad \langle f, g \rangle = \int fg \; d\mu.
\end{equation*}
We consider the following dyadic model of the Hilbert transform:
$$
\Hd f (x) = \sum_{I \in \D} \left[  \langle f, h_{I_+} \rangle h_{I_-} (x)- \langle f, h_{I_-} \rangle h_{I_+} (x)  \right].
$$
One can see that for each $I\in \D$,
\begin{equation*}
   \Hd(h_{I_{+}}):=h_{I_{-}} \text{ and},\,\Hd(h_{I_{-}}):=-h_{I_{+}}, 
\end{equation*}
so the behavior of the operator on discrete sine and cosine waves matches that of $\Hilbert$ on the continuous ones. $\Hd$ was studied in \cite{domelevopetermichl} and \cite{DKPS23}, where its connection with $\Hilbert$ is made evident. In fact, 
in the doubling setting, it was shown in [\cite{DKPS23},Theorem 2.1]] that the  
commutator $[\Hd, b]$ with this Haar shift characterizes dyadic $\BMO$ and $\Hd$ was used in their characterization of the multiparameter ``little $\BMO$" space. Further, the operator $\Hd$ has been useful in the recent proof of the counterexample to the so-called matrix $A_2$ conjecture \cite{DPTV}.  

By orthonormality of the Haar functions, $\Hd$ is an $L^2(\mu)$-bounded Haar shift operator. Thus the theory in \cite{LSMP} applies to it as well; in particular, it extends to an $L^p(\mu)$-bounded operator whenever $\mu$ is balanced. But the balanced class does not characterize the $L^p$-boundedness of this particular Haar shift, $\Hd$. We call $\mu$ {\it sibling balanced} if 
$$
\musib := \sup_{I\in \D} \frac{m(I)}{m(I^s)}< \infty.
$$
Clearly, balanced measures are sibling balanced. Moreover, we shall see that the sibling balanced class is strictly larger than the balanced one and is the right one to guarantee boundedness on $L^p(\mu)$ of $\Hd$ as well as that of the commutator $[\Hd,b]$.

Therefore, we test the commutator characterization of $\BMO(\mu)$ in that context. Perhaps surprisingly, even for this restricted class of measures, the characterization fails.

\begin{ltheorem} \label{th:theoremA}
   Let $\mu$ be sibling balanced, $b\in \BMO(\mu)$ and $1<p<\infty$. Then, the following hold: 
   \begin{itemize}
        \item \textbf{Upper estimate:}
$$
\|[\Hd,b]\|_{L^p(\mu)\rightarrow L^p(\mu)} \lesssim_{p,[\mu]_{sib}} \|b\|_{\BMO(\mu)}.
$$
        \item \textbf{Lower estimate:}
$$
\|b\|_{\BMO(\mu)}\lesssim_{p,[\mu]_{sib}} \|[\Hd,b]\|_{L^p(\mu)\rightarrow L^p(\mu)} + \sup_{k \in \Z} \|[\Hd,\cexp_k b]\|_{L^p(\mu)\rightarrow L^p(\mu)}.
$$
        \item \textbf{Failure of lower estimate in general:} The estimate 
$$
\|b\|_{\BMO(\mu)}\lesssim \|[\Hd,b]\|_{L^{2}(\mu)\rightarrow L^2(\mu)}
$$
fails in general, even if we allow the implicit constant to depend on $\musib$.         
    \end{itemize}
\end{ltheorem}

Straightforward calculations show that for any $k \in \Z$,
$ 
\|\cexp_k b\|_{\BMO(\mu)} \leq \|b \|_{\BMO(\mu)},
$
which suggests that it might be possible to remove the terms involving $[\Hd,\cexp_k b]$ in the lower estimate of Theorem \ref{th:theoremA}. However, we construct a family of sibling balanced measures and BMO functions that show that these bounds are essential to control the norm in $\BMO(\mu)$. The upper estimate ensuring the boundedness on $L^p(\mu)$ of $[\Hd,b]$ is, by contrast, the expected one. 

We next turn our attention to weighted estimates for the commutator, that is, boundedness with respect to measures of the form $\omega d\mu$ where $\omega$ is nonnegative. Both one and two-weight inequalities for commutators with singular integral operators have a long history. In the special case of the Hilbert transform and Lebesgue measure, Bloom (\cite{Bloom}) proved an inequality of the form $$\|[b, \mathbb{H}]\|_{L^p(\omega) \rightarrow L^p(\sigma)} \lesssim \|b\|_{{\rm BMO}_\nu},$$ where $\omega, \sigma \in A_p(\R)$, $\nu=\left( \frac{\omega}{\sigma} \right)^{1/p}$, and ${\rm BMO}_{\nu}$ denotes a weighted BMO space. When $\omega=\sigma$, Bloom's result reduces to the one weight inequality with a BMO-norm bound:  

$$\|[b, \mathbb{H}]\|_{L^p(\omega) \rightarrow L^p(\omega)} \lesssim \|b\|_{{\rm BMO}}, \quad \omega \in A_p,$$
and the implicit constant depends on $[\omega]_{A_p}.$ According to the authors of \cite{ChungPereyraPerez}, this one weight inequality was also implicit in a proof due to Str\"{o}mberg of the Coifman, Rochberg, and Weiss result (although the literature is slightly unclear on this point). Various generalizations of this one weight result have been considered in the more general context of Calder\'{o}n-Zygmund singular integral operators; see for example \cite{AlvarezBagbyKurtzPerez,ChungPereyraPerez}. Therefore, in the Euclidean (doubling) setting, we see that the Muckenhoupt class $A_p$ is the right one for weighted inequalities of commutators with singular integrals.  The sharp dependence of the operator norm of the commutator on the weight characteristic has also been investigated; see \cite{Chung} for the special case of the Hilbert transform and \cite{ChungPereyraPerez} for general Calder\'{o}n-Zygmund operators. Heuristically, the commutator behaves worse and is ``more singular'' than classical Calder\'{o}n-Zygmund operators, as evidenced by the sharp quadratic dependence on the $A_2$ weight characteristic - as opposed to the linear dependence proved in the $A_2$ theorem for the operators themselves.

The class $A_p^{bal}(\mu)$, $1<p<\infty$, was introduced in \cite{CAPW} to characterize the weighted bounds for Haar shifts when the underlying assumption that $\mu$ is balanced. We say that $\omega\in A_p^{bal}(\mu)$ if 

\begin{equation*}
[\omega]_{A_p^{bal}(\mu)}:=\sup_{\substack{I, J \in \D:\\ J \in \{I, I_{-}^s, I_{+}^s\} \text{ or }\\ I \in \{J,J_{-}^s, J_{+}^s\} }} \frac{m(I)^{p/2} m(J)^{p/2}}{\mu(I)^{p-1} \mu (J) } \langle \omega \rangle_{I} \langle \sigma \rangle_{J}^{p-1}< \infty,\label{A2b}
\end{equation*}
denoting the dual weight $\sigma=\omega^{-1/(p-1)}$. In view of Theorem \ref{th:theoremA}, one would not expect $A_p^{bal}(\mu)$ to be the right class of weights to study the boundedness of $\Hd$ and hence of $[\Hd,b]$. Below we explain the modifications needed to define the class $A_p^{sib}(\mu)$ that governs the $L^p$-boundedness of $\Hd$ with respect to an underlying sibling balanced measure. But the situation with the commutator is more interesting: one needs a characterization of $\BMO(\mu)$ functions in terms of weights $\omega$ (specifically, of their logarithms). To that end, we will see that neither of the classes $A_2^{bal}(\mu)$ and class $A_2^{sib}(\mu)$ are the right ones. Instead, $\text{BMO}(\mu)$ functions can be characterized by a smaller weight class $\Ahat_p(\mu)$, which is the class of weights $\omega$ which satisfy
\begin{equation}\label{A2s} \tag{$\widehat{A}_p(\mu)$}
[\omega]_{\widehat{A}_p(\mu)}:=\sup_{\substack{I, J \in \D:\\ J \in \{I, \widehat{I}\} \text{ or }\\ I \in \{ J, \widehat{J} \}  }}\langle\omega \rangle_{I} \langle \sigma \rangle_{J}^{p-1}< \infty. 
\end{equation}
The class $\widehat{A}_p(\mu)$ has several nice features that make it the right counterpart of the classical $A_p$ one to study commutators, partly because weights in this class satisfy a reverse H\"older inequality. 

\begin{ltheorem} \label{th:theoremB} Let $\mu$ be a locally finite Borel measure on $\R$ satisfying $\mu(I)>0$ for all $I \in \D$.
    \begin{itemize}
        \item[\textbf{(a)}] Suppose $\omega \in \widehat{A}_2(\mu). $ Then the function $\log \omega \in \BMO(\mu)$. Conversely, suppose $b \in \BMO(\mu).$ Then for sufficiently small $\delta>0$, the weight function $e^{\delta b}$ belongs to $\widehat{A}_2(\mu)$. 
        \item[\textbf{(b)}] Let $1<p<\infty$ and $\omega \in \widehat{A}_p(\mu)$. Then there exists $\gamma>1$, depending only on $p$ and $[\omega]_{\widehat{A}_p(\mu)}$, so that for all $I \in \D$,
\begin{equation*}
\left( \frac{1}{\mu(I)}\int_{I} \omega^\gamma \, d\mu \right)^{1/\gamma} \lesssim_{\omega} \left( \frac{1}{\mu(I)}\int_{I} \omega\, d\mu \right).
\end{equation*}
    \end{itemize}
\end{ltheorem}

It is important to remark that the class $\widehat{A}_p(\mu)$ seems to be the right nondoubling analog of $A_p$ for all nondoubling measures, not just the balanced or sibling balanced ones. Theorem \ref{th:theoremB} holds for all measures $\mu$, yielding a satisfactory theory of weights independent of the doubling condition. Weights in $\widehat{A}_p(\mu)$ are sufficient to ensure weighted $L^p$-boundedness of $[\Hd,b]$. 

\begin{ltheorem} \label{th:theoremC}
Let $\mu$ be sibling balanced and atomless. Let $1<p<\infty$, $b \in \BMO(\mu)$, and $\omega \in \widehat{A}_p(\mu)$. Then there holds 
\begin{equation*}\|[\Hd,b]\|_{L^p(\omega)\rightarrow L^{p}(\omega)} \lesssim_{[\omega]_{\Ahat_p(\mu)},p,\musib} \|b\|_{\BMO(\mu)}.
\end{equation*}
\end{ltheorem}

We investigate how specialized are our results by studying commutators $[T,b]$ for general Haar shifts, assuming the balanced condition on the underlying measure. Using a decomposition of general Haar shifts into simpler ones |that was already used in \cite{CAPW} and that we record in Section \ref{sec1}| we obtain results analogous to the upper estimate in Theorem \ref{th:theoremA} and Theorem \ref{th:theoremC} for Haar shifts in the general case. 

In Section \ref{sec0}, we prove the connection between sibling balanced measures and the boundedness of $\Hd$. In Section \ref{sec1}, we prove Theorem \ref{th:theoremA} and in Section \ref{secweights}, we explore the connections between the new weight classes and characterizations of $\BMO$.  We end the paper with an appendix which gives a sparse domination result for $\mathcal{H}$, similar to that of \cite{CAPW}, but with the balanced condition on the measure replaced by the weaker sibling balanced condition. This
result implies both the unweighted $L^p(\mu)$ and weak-type bounds and the weighted bounds for $\mathcal{H}$ with the appropriate class of weights, $A_p^{sib}(\mu)$. 

\subsection*{Acknowledgement} We would like to thank Sergei Treil for many valuable discussions which contributed to this paper, in particular our understanding of the counterexample that completes the proof of Theorem \ref{th:theoremA}.

\section{Characterizing the boundedness of the dyadic Hilbert transform }\label{sec0}

We first turn our attention to the relationship between the balanced and sibling balanced conditions. Let $\mu$ be a Borel measure on $\R$. Here and thereafter, for any sibling balanced measure $\mu$, we assume for simplicity that $0<\mu(I)<\infty$ for each $I \in \D$ and 
$$
\mu[0,\infty)=\infty=\mu[-\infty,0).
$$ The condition that $\mu$ is sibling balanced is weaker than having $\mu$ balanced since if $\mu$ is balanced then
$m(I)\sim m(\hat{I})\sim m(I^s)$ for all $I\in \D$. 
However, if $\mu$ is sibling balanced, then automatically $\mu$ is m-increasing. Recall that a Borel measure $\mu$ is m-increasing with respect to the dyadic grid $\D$ if $m(I) \lesssim m(\hat{I})$ for all $ I \in \D.$ The m-increasing condition characterizes the weak $(1,1)$ bound for the different dyadic analog of the Hilbert transform considered in \cite{LSMP}. However, the reverse implication fails. Moreover, there exists a sibling balanced measure $\mu$ which is not balanced. We summarize these observations in the following proposition:

\begin{prop}\label{sibbalanceincreas}
Let $\mu$ be sibling balanced. Then $\mu$ is m-increasing with $m(I)\lesssim [\mu]_{sib} m(\hat{I})$ for all $I\in \D$. On the other hand, there exists an m-increasing measure $\mu$ that is not sibling balanced. Moreover, there exists a sibling balanced measure that is not balanced.
\begin{proof}
For the first part, fix $I \in \D$ and observe that
\begin{align*}
m(\hat{I}) & \sim \min\{ \mu(I), \mu(I^s))\}\\
&  \geq\min \{ m(I), m(I^s) \}\\
& \geq \frac{1}{[\mu]_{sib}}\, m(I),
\end{align*}
where we used the fact that $\mu$ is sibling balanced in the last step. 

To construct a measure on $[0,1]$ that is m-increasing but not sibling balanced, we use the second example given in Section 4 of \cite{LSMP} and leave the verification to the reader. 



To exhibit a measure which is sibling balanced but not balanced, we construct a measure which is a sum of an absolutely continuous measure and point masses. Define the following measure on $\mathbb{R}$:
\begin{equation}\label{sibnotbalanced}
    d \mu(x)= \left(\sum_{k \geq 0} 2^{-k} (\delta_{2^k}(x) + \delta_{\frac{3}{2} \cdot 2^{k}}(x))\right)+ w(x) \, dx 
\end{equation}
where $\delta_a$ denotes the Dirac delta mass at $x=a$ and $$ w(x)= \begin{cases}
1; & - \infty<x<1\\
2^{-3k}; & 2^k<x<2^{k+1}.
\end{cases}$$

Then one can check $\mu$ is sibling balanced, but it is not balanced because $m(I_k^s) \sim 2^{-k}$ while the children of $I^s_k$ satisfy $m(I_{k,1}^s) \sim 2^{-2k}. $ Notice that $\mu$ is m-increasing, as it must be. 
    
\end{proof}

\end{prop}
 We have the following characterization of $L^p$ and weak-type $(1,1)$ estimates for $\mathcal H.$

\begin{prop}\label{weak11dyadicH} The following are equivalent.
\begin{enumerate}
    \item $\Hd$ is bounded in $L^{p}(\mu)\rightarrow L^{p}(\mu)$ for all $1<p<\infty$;
    \item $\Hd$ is bounded in $L^{p}(\mu)\rightarrow L^{p}(\mu)$ for some $p\neq 2$;
    \item $\mu$ is sibling balanced;
    \item $\Hd$ is weak type $(1,1)$.
\end{enumerate}
\end{prop}

\begin{proof}[Proof of Proposition \ref{weak11dyadicH}]
The implication $(1) \Rightarrow (2)$ is trivial, while the implication $(4) \Rightarrow (1)$ follows from the fact that $\Hd$ is an isometry on $L^2$ and anti self-adjoint, interpolation, and duality. The implication $(2) \Rightarrow (3)$ can be seen directly from testing on Haar functions and using standard estimates. 

 The most difficult implication, $(3) \Rightarrow (4)$, can be obtained by following the proof in \cite[Theorem 2.11]{LSMP}. We remark that one can streamline the argument by making use of the Calder\'{o}n-Zygmund decomposition provided in \cite{CAPW}. Again, we leave the details to the interested reader.

\end{proof}



\section{Commutators for sibling balanced measures} \label{sec1}

\subsection{Martingale \texorpdfstring{$\BMO$}{Lg}} We fix a dyadic grid $\D$ that we may assume to be the standard one, that is, 
$$
\D=\{2^{-u}[m,m+1)\colon m\in \Z,u\in \Z\}.
$$
Given $I\in \D$ and $u \in \N$, $I^{(u)}$ denotes the interval in $\D$ that contains $I$ with length $2^u\ell(I)$, and $\D_u(I)=\{J\in \D\colon J^{(u)}=I\}$.

We denote the standard martingale $\BMO$ space associated to the filtration generated by $\D$, namely $\{\sigma(\D_k)\}_{k\in\Z}$, by $\BMO(\mu)$. Its norm is given by 
$$
\|b\|_{\BMO(\mu)}=\sup_{I\in \D}\frac{1}{\mu(I)}\int_{I} |b-\avbIhat|d\mu.
$$
This is the natural definition of $\BMO$ when the underlying martingale filtration is not regular, because it is the one that yields interpolation with the $L^p$-scale. Also, whenever $b\in \BMO(\mu)$ one has $|\avbI-\avbIhat|\leq \|b\|_{\BMO(\mu)}$ |something trivial in the regular case|, so that 
$$
|\avbIplus-\avbIminus|\leq 2\|b\|_{\BMO(\mu)}.
$$
John-Nirenberg inequality holds for $\BMO(\mu)$ (see \cite{Garsia}), taking the following form: there exist constants $C, \delta_0$ so that for all $b \in \BMO(\mu)$ with norm $1$ and $I \in \D$

\begin{equation}
\mu(\{x \in I: |b(x) - \langle b \rangle_{\widehat{I}}|>\alpha \}) \leq C e^{-\delta_0 \alpha} \mu(I),
\label{JNMart1}
\end{equation}
and also 

\begin{equation}
\mu(\{x \in I: |b(x) - \langle b \rangle_{I}|>\alpha \}) \leq C e^{-\delta_0 \alpha} \mu(I).
\label{JNMart2}
\end{equation}
\eqref{JNMart2} can be easily deduced from \eqref{JNMart1} by replacing $\widehat{I}$ by $I$ and $I$ by $I_{+}$ and $I_{-}$, respectively, giving us two inequalities which we can then sum. The John-Nirenberg inequality implies 
\begin{equation}\label{bmodef}
\begin{split}
\|b\|_{\BMO(\mu)} & \sim \sup_{I \in \D}\frac{1}{\mu(I)}\int_{I} |b-\langle b\rangle_I| d\mu + \sup_{I \in \D} |\avbIplus-\avbIminus|\\
& \sim \sup_{I \in \D} \left(\frac{1}{\mu(I)}\int_{I} |b-\langle b\rangle_I|^p d\mu\right)^{1/p} + \sup_{I \in \D} |\avbIplus-\avbIminus|\\
& =: K_b^p + D_b.
\end{split}
\end{equation} 

\subsection{Upper bounds} In the remainder of the paper, and unless otherwise stated, we assume that $\mu$ is a sibling balanced measure. Let $b \in L^1_{\mathrm{loc}}(\mu)$. Writing $b=\sum_{I\in \D}\langle b,h_I\rangle h_I$ and $f=\sum_{J\in \D}\langle f,h_I\rangle h_{J}$, we can decompose 
\begin{equation}
\begin{split}
    b(x)f(x)=\sum_{I,J\in \D} \langle b,h_I \rangle \langle f,h_J\rangle h_I(x) h_J(x)= \text{I}+ \text{II}+\text{III}
\end{split}
\end{equation}
where 
\begin{equation}\label{paraproductsdef}
    \begin{split}
\text{I}=&\sum_{I}\bhaarI\fhaarI h_I^2(x)=:\Delta_bf(x),\\
        \text{II}=& \sum_{I} \bhaarI h_I(x) (\sum_{J\colon J\supsetneq I}\fhaarJ h_J(x))=\sum_{I} \bhaarI \avfI h_I(x)=:\pi_{b}f(x),\\
        \text{III}=&\sum_{J} \fhaarJ h_J(x) (\sum_{I\colon I\supsetneq J}\bhaarI h_I(x))=\sum_{I} \fhaarI \avbI h_I(x)=\pi_{f}b(x)=:\Lambda_{b}^0f(x).
    \end{split}
\end{equation}

The paraproduct decomposition above appears in \cite{Treil13}. The operator that we denote $\Delta_b$ plays the role of the operator $\pi_b^{(*)}$ in their work. Using the decomposition $bf=\Delta_bf+\pi_bf+\Lambda_{b}^{0}f$, it follows that
\begin{equation} \label{paraproductdecomp}
\begin{split}
   [\Hd,b](f):=&\Hd(bf)-b(\Hd f)\\  
   =&[\Hd,\pi_b](f)+[\Hd,\Delta_b](f)+[\Hd,\Lambda_b^{0}](f).
\end{split}
\end{equation}

We will refer to term $R_b(f)=[\Hd,\Lambda_b^{0}](f)$ as the remainder term. The splitting above allows us to prove the upper estimate in Theorem \ref{th:theoremA}. 

\begin{proof}[Proof of Theorem \ref{th:theoremA}, upper estimate]
We start by controlling the remainder term $R_bf=[\Hd,\Lambda_{b}^{0}]$. To that end, we write 
\begin{equation*}
    \begin{split}
        \Hd(\Lambda_b^0 f)=&\Hd\left(\sum_{I\in \D} \fhaarI \avbI h_I\right)=\sum_{I\in \D}\fhaarI\avbI \text{sign}(I)h_{I^s},\\
        \Lambda_{b}^{0}(\Hd f)=&\sum_{I\in \D} \langle \Hd f,h_I\rangle\avbI h_I=\sum_{I\in \D} \fhaarIsib \text{sign}(I^s)\avbI h_I=\sum_{I\in \D} \fhaarI \avbIsib \text{sign}(I)h_{I^{s}}.
    \end{split}
\end{equation*}
Therefore,
\begin{equation}\label{expressionforRb}
\begin{split}
     R_bf=\Hd(\Lambda_b^0 f)-\Lambda_{b}^{0}(\Hd f)=\sum_{I\in \D} \alpha_I\fhaarI h_{I^s}
\end{split}
\end{equation}
for a sequence $\{\alpha_I=\text{sign}(I)(\avbI-\avbIsib)\}$ with $\sup_{I}|\alpha_I|\leq 2\|b\|_{BMO(\mu)}$, so $R_b$ is a Haar shift. Therefore, the results in \cite{LSMP} in the balanced case, and those in Appendix \ref{sec:app} in the |more general| sibling balanced one, imply that $R_b$ is of weak-type $(1,1)$ and bounded on $L^p(\mu)$, and moreover
$$
\|[\Hd,\Lambda_{b}^{0}]\|_{L^{p}(\mu)\rightarrow L^{p}(\mu)}\lesssim_{p,[\mu]_{sib}} \|b\|_{BMO(\mu)}.
$$
To control the other two pieces from \eqref{paraproductdecomp}, we apply \cite[Theorem 4.8]{Treil13}, which in particular implies  
$$\|\Delta_b\|_{L^{p}(\mu)\rightarrow L^{p}(\mu)}\sim \|b\|_{\BMO(\mu)}$$
and 
$$\|\pi_b\|_{L^{p}(\mu)\rightarrow L^{p}(\mu)}\lesssim \|b\|_{\BMO(\mu)}.$$
Combining this with the $L^p(\mu)$-boundedness of $\Hd$ finishes the proof. 
\end{proof}


\subsection{Failure of lower bounds} \label{subsec:13} The fact that $\Hd^2=-I$ and the $L^p$-boundedness of $\Hd$ imply
\begin{equation}\label{ScomparableLp}
\|\Hd f\|_{L^p(\mu)} \sim \|f\|_{L^p(\mu)}, \; 1<p<\infty.
\end{equation}
This is enough to conclude the lower estimate for $[\Hd, b]$ in the doubling setting. In our setting, one has the following:

\begin{prop}\label{lowerboundforclassicalbmo}
For $1<p<\infty$, 
$$
\sup_{I\in \D} \frac{1}{\mu(I)}\int_{I} |b-\langle b\rangle_I|^p d\mu \lesssim_{p,[\mu]_{sib}} \|[\Hd,b]\|_{L^p(\mu)\rightarrow L^p(\mu)}^p.
$$
\end{prop}

\begin{proof}
It is enough to check that for all $I\in \D$, one has
$$
\int_{I} |b-\langle b\rangle_I|^pd\mu\lesssim \|[\Hd,b] \one_{I} \|_{L^p(\mu)}^p.
$$
Clearly,
\begin{equation*}
    \|[\Hd,b]\one_{I}\|_{L^p(\mu)}^p\geq \int_{\hat{I}} |\Hd(b\one_I)-b\Hd(\one_I)|^pd\mu.
\end{equation*}
Now $\one_{I}=\sum_{J\colon I\subsetneq J} \langle \one_{I},h_J\rangle h_J$, and for all $J$ with $I\subsetneq J$ one has $\Hd(h_J)=0$ in $\hat{I}$ so $\Hd(\one_I)=0$ in $\hat{I}$. Therefore, the integral on the right hand side reduces to
$$\int_{\hat{I}} |\Hd(b\one_I)|^pd\mu=\int |\Hd(\sum_{J\subseteq I}\langle b,h_J \rangle h_J)|^p d\mu\\ \gtrsim_{p,[\mu]_{sib}} \int |\sum_{J\subseteq I}\langle b,h_J \rangle h_J|^p d\mu
$$
by \eqref{ScomparableLp}, and so 
$$\|[\Hd,b]\one_I\|_{L^p(\mu)}^p \gtrsim_{p,[\mu]_{sib}} \int |\sum_{J\subseteq I}\langle b,h_J \rangle h_J|^p d\mu=\int_{I} |b-\langle b\rangle_I |^pd\mu.$$
\end{proof}


\begin{proof}[Proof of Theorem \ref{th:theoremA}, lower estimate] In view of Proposition \ref{lowerboundforclassicalbmo}, we only need to show
$$
\sup_{I\in \D} |\avbIplus-\avbIminus| \lesssim_{p,[\mu]_{sib}} \sup_{k \in \Z} \|[\Hd,\cexp_k b]\|_{L^p(\mu)\rightarrow L^p(\mu)}.
$$
Fix $I \in \D$, and suppose $I \in \D_k$. We test $[\Hd, \cexp_{k+1}b]$  on $h_{I_{-}}$:

\begin{align*}
\|[\Hd,\cexp_{k+1} b]\|_{L^p(\mu)\rightarrow L^p(\mu)} \|h_{I_{-}}\|_{L^p(\mu)} & \gtrsim \|\Hd(\cexp_{k+1}b h_{I_{-}})- \cexp_{k+1} b \Hd(h_{I_{-}})\|_{L^p(\mu)}\\
& = \|(\avbIminus- \avbIplus) h_{I_{+}}\|_{L^p(\mu)}\\
& = |\avbIminus- \avbIplus| \cdot \|\Hd(h_{I_{-}})\|_{L^p(\mu)} \\
& \gtrsim |\avbIminus- \avbIplus| \cdot \|h_{I_{-}}\|_{L^p(\mu)},
\end{align*}
by \eqref{ScomparableLp}. 
\end{proof}

\begin{rem}
As an immediate consequence of \eqref{expressionforRb}, we immediately get the alternative lower bound 
$$
\|b\|_{\BMO(\mu)}\lesssim_{p,[\mu]_{sib}} \|[\Hd,b]\|_{L^{p}(\mu)\rightarrow L^{p}(\mu)}+\|[\Hd,\Lambda_{b}^{0}]\|_{L^{p}(\mu)\rightarrow L^{p}(\mu)}.
$$
\end{rem}

To complete the proof of Theorem \ref{th:theoremA}, we will show that an inequality like 
\begin{equation}\label{faillowerbound}
    \|b\|_{\BMO(\mu)}\leq C \|[\Hd,b]\|_{L^{2}(\mu)\rightarrow L^2(\mu)}
\end{equation}
cannot hold for all sibling balanced measures $\mu$ with $[\mu]_{sib}\sim 1$ and all $b\in BMO(\mu)$. For each $k>0$, define the measure
    \begin{equation}\label{defmeasure}
    d\mu_k(x)=\sum_{\ell\in \Z}\delta_\ell(x)
    + 2^{-k}\one_{\R\backslash \Z}(x)dx,
    \end{equation}
and put 
    \begin{equation}\label{defb}
    b_k(x)=\begin{cases}
        1,\text{ if }x\in \Z;\\
        2^k, \text{ if }x\in\R\backslash \Z.
    \end{cases}
    \end{equation}

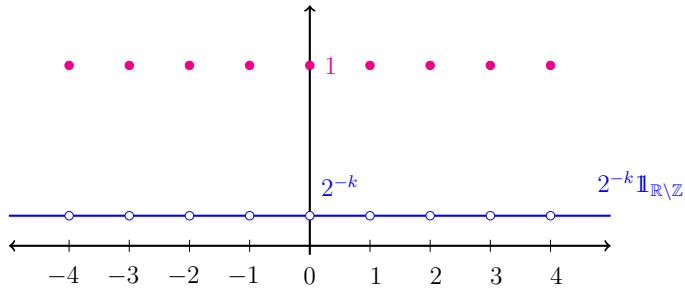
\begin{figure}[h]\label{figure2}
         \scalebox{0.8}{
\begin{tikzpicture}
\draw[<->,line width=1pt] (-5,0)--(5,0);
\draw[->,line width=1pt] (0,-0.15)--(0,4);

\draw (0,-0.5) node {$0$};

\draw[blue, line width=1pt] (-5,1/2)--(5,1/2);
\draw[blue] (5.5,1) node {$2^{-k}\one_{\R\backslash \Z}$};
\draw[blue, fill=white] (1,1/2) circle (2pt);
\draw[blue, fill=white] (2,1/2) circle (2pt);
\draw[blue, fill=white] (3,1/2) circle (2pt);
\draw[blue, fill=white] (4,1/2) circle (2pt);
\draw[blue, fill=white] (-1,1/2) circle (2pt);
\draw[blue, fill=white] (-2,1/2) circle (2pt);
\draw[blue, fill=white] (-3,1/2) circle (2pt);
\draw[blue, fill=white] (-4,1/2) circle (2pt);
\draw[blue, fill=white] (0,1/2) circle (2pt);

\draw[magenta] (0.35,3) node {$1$};
\draw[blue] (0.5,1) node {$2^{-k}$};

\filldraw[magenta] (1,3) circle (2pt);
\filldraw[magenta] (2,3) circle (2pt);
\filldraw[magenta] (3,3) circle (2pt);
\filldraw[magenta] (4,3) circle (2pt);
\filldraw[magenta] (0,3) circle (2pt);
\filldraw[magenta] (-1,3) circle (2pt);
\filldraw[magenta] (-2,3) circle (2pt);
\filldraw[magenta] (-3,3) circle (2pt);
\filldraw[magenta] (-4,3) circle (2pt);

\draw (-4,-0.1)--(-4,0.1);
\draw (-3,-0.1)--(-3,0.1);
\draw (-2,-0.1)--(-2,0.1);
\draw (-1,-0.1)--(-1,0.1);
\draw (0,-0.1)--(0,0.1);
\draw (1,-0.1)--(1,0.1);
\draw (2,-0.1)--(2,0.1);
\draw (3,-0.1)--(3,0.1);
\draw (4,-0.1)--(4,0.1);

\draw (-4.1,-0.5) node {$-4$};
\draw (-3.1,-0.5) node {$-3$};
\draw (-2.1,-0.5) node {$-2$};
\draw (-1.1,-0.5) node {$-1$};
\draw (4.1,-0.5) node {$4$};
\draw (3.1,-0.5) node {$3$};
\draw (2.1,-0.5) node {$2$};
\draw (1.1,-0.5) node {$1$};

\end{tikzpicture}}
\caption{Illustration of the measure $\mu$ for a fixed $k\in \N$. The circles represent the point masses at the integers.}

\end{figure}

\begin{figure}[h]\label{figure3}
         \scalebox{0.8}{
\begin{tikzpicture}
\draw[<->,line width=1pt] (-5,0)--(5,0);
\draw[->,line width=1pt] (0,-0.15)--(0,4.5);
\draw (0,-0.5) node {$0$};

\draw[blue, line width=1pt] (-5,4)--(5,4);
\draw[blue] (5.75,4) node {$2^{k}\one_{\R\backslash \Z}$};
\draw[blue, fill=white] (1,4) circle (2pt);
\draw[blue,fill=white] (2,4) circle (2pt);
\draw[blue,fill=white] (3,4) circle (2pt);
\draw[blue,fill=white] (4,4) circle (2pt);
\draw[blue,fill=white] (-1,4) circle (2pt);
\draw[blue,fill=white] (-2,4) circle (2pt);
\draw[blue,fill=white] (-3,4) circle (2pt);
\draw[blue,fill=white] (-4,4) circle (2pt);
\draw[blue,fill=white] (0,4) circle (2pt);

\draw[magenta] (0.35,1) node {$1$};
\draw[blue] (0.35,3.7) node {$2^{k}$};
\filldraw[magenta] (1,1) circle (2pt);
\filldraw[magenta] (2,1) circle (2pt);
\filldraw[magenta] (3,1) circle (2pt);
\filldraw[magenta] (4,1) circle (2pt);
\filldraw[magenta] (0,1) circle (2pt);
\filldraw[magenta] (-1,1) circle (2pt);
\filldraw[magenta] (-2,1) circle (2pt);
\filldraw[magenta] (-3,1) circle (2pt);
\filldraw[magenta] (-4,1) circle (2pt);

\draw (-4,-0.1)--(-4,0.1);
\draw (-3,-0.1)--(-3,0.1);
\draw (-2,-0.1)--(-2,0.1);
\draw (-1,-0.1)--(-1,0.1);
\draw (0,-0.1)--(0,0.1);
\draw (1,-0.1)--(1,0.1);
\draw (2,-0.1)--(2,0.1);
\draw (3,-0.1)--(3,0.1);
\draw (4,-0.1)--(4,0.1);

\draw (-4.1,-0.5) node {$-4$};
\draw (-3.1,-0.5) node {$-3$};
\draw (-2.1,-0.5) node {$-2$};
\draw (-1.1,-0.5) node {$-1$};
\draw (4.1,-0.5) node {$4$};
\draw (3.1,-0.5) node {$3$};
\draw (2.1,-0.5) node {$2$};
\draw (1.1,-0.5) node {$1$};

\end{tikzpicture}}
\caption{Illustration of the function $b$ for a fixed $k\in \N$.}

\end{figure}
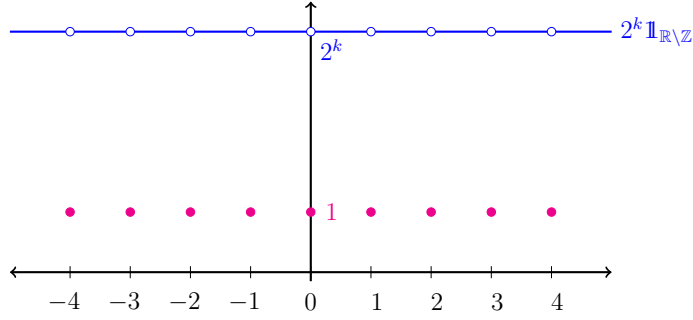

The next result summarizes the properties of the pairs $(\mu_k,b_k)$.

\begin{prop}\label{propertiesbmu} 
\begin{enumerate}
    \item \label{sibbalanced}$[\mu_k]_{sib} \sim 1$.
    \item\label{martnorm} $\sup_{I\in \D} |\langle b_k \rangle_I-\langle b_k \rangle_{I^s}|\sim 2^k$;
    \item\label{classnorm} $\sup_{I\in \D}\left( \frac{1}{\mu(I)}\int_I|b_k-\langle b_k\rangle_I |^2d\mu\right)^{1/2}\lesssim 2^{k/2}$.
\end{enumerate}
\end{prop}
\begin{proof} We start checking \eqref{sibbalanced}. If $I\in \D_{\leq 0}$, $\mu_k(I_{-})=\mu_k((I^s)_{-})$, $\mu_k(I_{+})=\mu_k((I^s)_{+})$ and $\mu_k(I)=\mu_k(I^{s})$, so it is immediate that $m(I)=m(I^s)$. Define 
$$
\mathcal{I}=\{I_\ell^j=[\ell,\ell+2^{-j})\colon \ell\in \Z,\, \text{ and }j\geq 0\},
$$
and 
$$
\mathcal{I}_{>0}=\{I_\ell^j=[l,\ell+2^{-j})\colon \ell\in \Z,\, \text{ and }j\geq 1\}.
$$
If $|I|<1$, then the equality is still true if neither $I$ nor $I^{s}$ belong to $\mathcal{I}_{>0}$. We are then left to check $m(I)\sim m(I^s)$ in the case $I\in \mathcal{I}_{>0}$ or $I^{s}\in \mathcal{I}_{>0} $. If $I=I^{j}_l$ then
$$
m(I)\sim \mu_k(I_{+})\sim 2^{-k}|I|\sim m(I^{s}).
$$
If $I^{s}\in \mathcal{I}_{>0}$ we just replace $I$ with $I^s$ in the computation above. Hence, $\mu_k$ is sibling balanced with constant independent of $k$. 
    
We next check \eqref{martnorm}. Again, if $I\in \D_{\leq 0}$ then $\mu_k(I)=\mu_K(I^{s})$, and so since $b$ is 1 periodic then $\langle b_k \rangle_I=\langle b_k \rangle_{I^s}$. If $|I|<1$ then $\langle b_k \rangle_I=\langle b_k \rangle_{I^s} =2^k$ unless $I$ or $I^{s}$ belong to $\mathcal{I}_{>0}$. Therefore, 
$$
\sup_{I\in \D} |\langle b_k \rangle_I-\langle b_k \rangle_{I^s}|=\sup_{I\in \mathcal{I}_{>0}}|\langle b_k \rangle_I-\langle b_k \rangle_{I^s}|.$$
If $I=I_{\ell}^{j}$, $j\geq 1$, one has $\langle b_k \rangle_{I^s}=2^k$ while
\begin{equation*}
        \langle b_k \rangle_I=\langle b_k\rangle_{[0,2^{-j})}=\frac{b_k(0)+2^k\mu(0,2^{-j})}{\mu([0,2^{-j})])}=\frac{1+2^{-j}}{1+2^{-k-j}}\sim 1
\end{equation*}
so that $|\langle b_k \rangle_I-\langle b_k \rangle_{I^s}|\sim 2^k$ for all $I\in \mathcal{I}_{>0}$, as desired.

We are left with \eqref{classnorm}. From the fact that $\langle b_k \rangle_I-\langle b_k \rangle_{I^s}=0$ if neither $I$ nor $I^s$ belong to $\mathcal{I}_{>0}$, we also know $\langle b_k,h_I\rangle=\sqrt{m(I)}(\langle b_k \rangle_{I_-}-\langle b_k \rangle_{I_+})=0$ unless $I\in \mathcal{I}$.
Therefore, it is enough to check that for all $j\geq 0$ and $\ell\in \Z$,
$$\sum_{I\in \mathcal{I}\colon I\subset I_{\ell}^{j}}|\langle b_k,h_I\rangle|^2\lesssim 2^{k} \mu(I_{\ell}^{j}).$$
For any $I\in \mathcal{I}$ one has 
\begin{equation*}
    \begin{split}
    |\langle b_k,h_I\rangle|=&\sqrt{m(I)}|\langle b_k\rangle_{I_{-}}-\langle b_k \rangle_{I_{+}}|\\
    \sim & (|I|2^{-k})^{1/2}2^k\\
    \sim & |I|^{1/2}2^{k/2}.
    \end{split}
\end{equation*}
Hence,
$$\sum_{I\in \mathcal{I}\colon I\subset I_{\ell}^{j}}|\langle b_k,h_I\rangle|^2\lesssim 2^k\sum_{I\in \mathcal{I}\colon I\subset I_{l}^{j}}|I|\lesssim 2^{k}\sum_{m\geq 0} 2^{-m-j}\lesssim 2^{k-j}\leq 2^k{\mu(I_\ell^j)}.$$
\end{proof}

To finish the proof of Theorem \ref{th:theoremA}, we prove the following estimate:

\begin{prop} \label{prop:finalcommutator}
For each $k>0$,
$$
\|[\Hd,b_k]\|_{L^{2}(\mu)\rightarrow L^{2}(\mu_k)}\lesssim 2^{\frac{k}{2}}.
$$
\end{prop}

Indeed, Proposition \ref{prop:finalcommutator} and \eqref{martnorm} of Proposition \ref{propertiesbmu} are enough to conclude by taking $k$ arbitrarily large. To prove Proposition \ref{prop:finalcommutator}, we are going to consider an alternative paraproduct decomposition. As in \cite{Treil13}, we start with
$$
bf=\sum_{I}\langle bf,h_I\rangle h_I=\sum_{I,J} \fhaarJ \langle bh_J,h_I\rangle h_I.
$$
Splitting the sum into the pieces $I\subsetneq J,\,J\subsetneq I$ and $I=J$, we get the decomposition 
\begin{equation}\label{secondsplitting}
    bf=\pi_b f+\pi^{*}_b f+\Lambda_{b}f,
\end{equation}
where 
\begin{equation}\label{deflambdab}
    \begin{split}
        \Lambda_{b}f=&\sum_{I\in \D}\fhaarI \langle b,h_I^2\rangle h_I\\
=&\Lambda_b^{0}f+\Lambda_b^{1}f,
    \end{split}
\end{equation}
and
\begin{equation}\label{deflambdas}
  \Lambda_{b}^{0}f(x)=\sum_{I} \fhaarI \avbI h_I(x),\,\text{and }\Lambda_{b}^{1}f(x)=\sum_{I\in \D}\fhaarI \langle b_k,h_I\rangle \left(\int h_I^3 d\mu\right) h_I(x).  
\end{equation}

\begin{proof}[Proof of Proposition \ref{prop:finalcommutator}]
The above decomposition leads to
$$
[\Hd,b_k]=[\Hd,\pi_{b_k}]+[\Hd,\pi_{b_k}^{*}]+[\Hd,\Lambda_{b_k}].
$$
By \cite[Theorem 4.8]{Treil13}, $\|\pi_{b_k}\|_{L^{2}(\mu_k)\rightarrow L^2(\mu_k)}\lesssim K_{b_k} \sim 2^{k/2}$, and the same holds for the adjoint $\pi_{b_k}^{*}$. We claim 
$$\|[\Hd,\Lambda_{b_k}]\|_{L^2(\mu_k)\rightarrow L^2(\mu_k)}\lesssim 1,
$$
which is enough to conclude. First, we observe that 
$$\Lambda_{b_k}f=\sum_{I} \fhaarI \left(\langle b_k,h_I\rangle \int h_I^3 d\mu_k+\langle b_k \rangle_I \right)h_I=:\sum_{I} \fhaarI c_{I}(b_k)h_I.$$
One has
\begin{equation*}
    \begin{split}
        [\Hd,\Lambda_{b_k}]f=&\Hd(\Lambda_{b_k} f)-\Lambda_{b_k}(\Hd f)\\
        =&\sum_{I} \fhaarI c_I(b_k)\text{sign}(I) h_{I^s}-\Lambda_{b_k} (\sum_I \fhaarI \text{sign(I)}h_{I^s})\\
        =&\sum_{I} \fhaarI c_I(b_k)\text{sign}(I) h_{I^s}- \sum_I \fhaarIsib \text{sign}(I^s)c_I(b_k)h_{I}\\
        =&\sum_{I} \fhaarI \text{sign}(I)(c_{I}(b_k)-c_{I^s}(b_k))h_{I^s}.
    \end{split}
\end{equation*}
We next claim that for all $I\in \D$ and any measure $\mu$, one has
\begin{equation}\label{lemmacubedhaar}
    \sqrt{m(I)}\int h_{I}^3(x) d\mu(x) =\frac{\mu(I^{+})-\mu(I^{-})}{\mu(I)}.
\end{equation}
Indeed,  
    \begin{equation*}
        \begin{split}
            m(I)^{1/2}\int h_I^3d\mu=&m(I)^2\left(\int_{I_{-}} \frac{1}{\mu(I_{-})^3}d\mu-\int_{I_{+}}\frac{1}{\mu(I_{+})^3}d\mu\right)\\
            =&\frac{\mu(I_{+})^2\mu(I_{-})^2}{\mu(I)^2}\left(\frac{1}{\mu(I_{-})^2}-\frac{1}{\mu(I_{+})^2}\right)\\
            =&\frac{\mu(I_{+})^2-\mu(I_{-})^2}{\mu(I)^2}=\frac{\mu(I_{+})-\mu(I_{-})}{\mu(I)}.
        \end{split}
    \end{equation*}
From \eqref{lemmacubedhaar}, if neither $I$ nor $I^s$ are in $\mathcal{I}$, then $(c_{I}(b_k)-c_{I^s}(b_k))=\langle b_k \rangle_{I} -\langle b_k \rangle_{I^s}=0.$ If $I=I_\ell^j\in \mathcal{I}$, then $\langle b_k \rangle_{I} +\langle b_k \rangle_{I^s}\lesssim 1$ so
    \begin{equation}
        \begin{split}
            c_I(b)=&\langle b_k,h_I\rangle\left(\int h_I^3d\mu_k\right)  +\langle b_k \rangle_{I} \\
            =& (\langle b_k\rangle_{I_{-}}-\langle b_k\rangle_{I_{+}})\frac{\mu_k(I_{+})-\mu_k(I_{-})}{\mu_k(I)}+O(1)\\
            =&\langle b_k\rangle_{I_{+}}\frac{\mu_k(I_{-})-\mu_k(I_{+})}{\mu_k(I)}+O(1)\\
            =&2^k \frac{\mu_k(I_{-})}{\mu_k(I)}+O(1)=2^{k}+O(1),
        \end{split}
    \end{equation}
    where in the last equality we used that ${\mu_k(I_{+})2^k}\sim |I|2^{-k}2^k\lesssim 1$. On the other side we have either $c_{I^s}(b_k)=2^k$ if $j\geq 1$, or $c_{I^s}(b_k)=c_I(b_k)$ if $j=0$. In any case, 
    $$|c_I(b_k)-c_{I^s}(b_k)|\lesssim 1.$$
\end{proof}

\begin{rem}
We leave it as an interesting open question whether one can construct a \emph{single} sibling balanced measure $\mu$ and function $b$ so that the following properties are satisfied:
\begin{enumerate}
    \item $\mu$ is sibling balanced;
    \item $b$ satisfies the Carleson packing condition but does not satisfy $b \in \BMO(\mu)$;
    \item The commutator $[\Hd, b]$ is bounded on $L^2(\mu).$
\end{enumerate}

In light of our previous computations in this section we expect that such a construction can be carried out.
\end{rem}

\subsection{General Haar shifts} \label{subsec14} We next sketch how the upper bound in Theorem \ref{th:theoremA} can be extended to other Haar shift operators. Assume $\mu$ is balanced and let $\alpha:=\{\alpha_{J,K}^{I}\}$ be a bounded  sequence of complex numbers and $u,v \in \N.$ We define the Haar shift $T^{u,v,\alpha}$ of complexity $(u,v)$ as

\begin{equation*}
T^{u,v,\alpha}f:= \sum_{I \in \D} \sum_{J \in \D_u(I)} \sum_{K \in \D_v(I)} \alpha_{J,K}^{I} \langle f, h_J \rangle h_K, \quad f \in L^2(\mu).  
\end{equation*}
A general Haar shift can be decomposed into simpler operators as follows. Given a dyadic interval $I=[2^{-k}p,2^{-k}(p+1))$ and an integer $u>0$, let 
$$
I_{u}^m=[2^{-k}p+ 2^{-k-u}(m-1),2^{-k}p+ 2^{-k-u}m), \; 1 \leq m \leq 2^u.
$$
Fix $u,v\geq 0$ and $1 \leq m \leq 2^u$, $1 \leq n \leq 2^v$, and consider the operator 
$$ 
T_{m,n}^{u,v,\alpha}(f)= \sum_{I \in \D} \alpha^{I}_{I_u^m, I_v^n} \langle f, h_{I_u^m} \rangle h_{I_v^n}, \quad f \in L^2(\mu),
$$
Clearly,
$$
T^{u,v,\alpha} = \sum_{m=1}^{2^{u}} \sum_{n=1}^{2^v} T_{m,n}^{u,v,\alpha}.
$$
Because of that, it is enough to obtain the estimate
$$
\left\|[T_{m,n}^{u,v,\alpha},b] \right\|_{L^p(\mu)\to L^p(\mu)} \lesssim \|b\|_{\BMO(\mu)},
$$
uniformly on $m,n$ to get it for a general $T^{u,v,\alpha}$. Call one such operator $T$. We can apply decomposition \eqref{paraproductdecomp} to get
\begin{equation*}
[T,b](f)=[T,\pi_b](f)+[T,\Delta_b](f)+(T\pi_fb-\pi_{Tf}b). 
\end{equation*}
The $L^p(\mu)$ bound for $T$ from \cite{LSMP} implies that we only have to handle the remainder $R_bf=T\pi_fb-\pi_{Tf}b.$ We compute

 \begin{align*}
 T\pi_fb-\pi_{Tf}b & = \sum_{I \in \D}  \langle f, h_I \rangle \langle b \rangle_{I} T(h_{I}) -\sum_{I \in \D} \langle Tf, h_I \rangle \langle b \rangle_{I} h_I\\
 & = \sum_{\substack{I \in \D:\\ \exists J \in \D \text{ so }I=J_u^m }} \alpha^{J}_{J_u^m, J_v^n}   \langle f, h_I \rangle \langle b \rangle_{I} h_{J_v^n}- \sum_{\substack{I \in \D:\\ \exists J \in \D \text{ so }I=J_v^n }} \alpha^{J}_{J_u^m, J_v^n}   \langle f, h_{J_u^m} \rangle \langle b \rangle_{I} h_{I}\\
 & = \sum_{J \in \D} \alpha^{J}_{J_u^m, J_v^n}  \langle f, h_{J_u^m} \rangle \langle b \rangle_{J_u^m} h_{J_v^n}- \sum_{J \in \D} \alpha^{J}_{J_u^m, J_v^n}    \langle f, h_{J_u^m} \rangle \langle b \rangle_{J_v^n} h_{J_v^n}\\
 & =  \sum_{J \in \D} \beta^{J}_{J_u^m, J_v^n} \langle f, h_{J_u^m} \rangle  h_{J_v^n},
 \end{align*}
 where $\beta^{J}_{J_u^m, J_v^n}=\alpha^{J}_{J_u^m, J_v^n} (\langle b \rangle_{J_u^m}- \langle b \rangle_{J_v^n})$. It is clear that

 $$ \sup_{J \in \D}|\beta^{J}_{J_u^m, J_v^n}| \leq (u+v) \|b\|_{\text{BMO}(\mu)},$$
 whereby the estimate 
 $$ \|R_b f \|_{L^p(\mu)} \lesssim \|b\|_{\text{BMO}(\mu)}\|T\|_{L^{p}(\mu)\rightarrow L^{p}(\mu) } \|f\|_{L^p(\mu)}$$
 easily follows. 

 We emphasize that this upper bound for the $L^p$ norm of the commutator $[T,b]$ in terms of the BMO norm seems to be new even for the  case where $T$ is the classical dyadic Hilbert transform given by $\mathcal{S}f=\sum_{I\in \D}\fhaarI (h_{I_{-}}-h_{I_{+}})$. Endpoint behavior results for $[\mathcal{S},b]$ can be found in \cite[Corollary 5.9]{Bonami2023} under the assumption that $\mu$ is $m$-increasing.
 
\section{Weights, commutators and \texorpdfstring{$\BMO$}{Lg}}\label{secweights}

\subsection{A new weight class} The class $\Ahat_p(\mu)$ is smaller in general than $A_p^{bal}(\mu)$ as we will show in Subsection \ref{subseccompareweights} for the particular case $p=2$ for simplicity. We take advantage of that to characterize $\BMO(\mu)$ in terms of weights. The results in this subsection hold for any locally finite measure $\mu$ which is non-zero on dyadic intervals; the balanced assumption plays no role in the analysis.  We begin the proof of 
Theorem \ref{th:theoremB}, starting with
the proof that the exponential of a $\BMO(\mu)$ function belongs to $\Ahat_p(\mu)$. 

\begin{lemma} \label{sufficient} Let $1<p<\infty.$ There exist independent constants $C_1, C_2$, so that if $b \in \BMO(\mu)$, and $0<\delta<\frac{C_1\min\{1, p-1\} }{\|b\|_{\BMO(\mu)}}$, $e^{\delta b} \in \widehat{A}_p(\mu)$ with 
$$
[e^{\delta b}]_{\widehat{A}_p(\mu)} \leq C_2.
$$ 
\end{lemma}

\begin{proof}
By rescaling, we may assume that $\|b\|_{\BMO(\mu)}=1.$ Let $\delta_0$ a parameter for which the John-Nirenberg inequality holds as in (\ref{JNMart1}) and (\ref{JNMart2}). We claim that for any $0<\delta \leq \frac{\delta_0 \min \{1, p-1\} }{2 }$ the function $e^{\delta b} \in \widehat{A}_p(\mu)$ with uniformly bounded $\widehat{A}_p(\mu)$ characteristic, by which the conclusion of the lemma follows. To this end, we estimate using \eqref{JNMart1}:

\begin{align*}
\int_{I} e^{\delta b(x)}\, d\mu(x) & = \int_{I} e^{\delta (b(x)- \langle b \rangle_{\widehat{I}})}\, d\mu(x) \cdot e^{\delta \langle b \rangle_{\widehat{I}}}\\
& = e^{\delta \langle b \rangle_{\widehat{I}}} \int_{0}^{\infty} \mu(\{x \in I: e^{\delta (b(x)- \langle b \rangle_{\widehat{I}})}>\alpha \}) \, d \alpha \\
& \leq e^{\delta \langle b \rangle_{\widehat{I}}} \left( \mu(I)+\int_{1}^{\infty} \mu(\{x \in I: e^{\delta |b(x)- \langle b \rangle_{\widehat{I}}|}>\alpha \}) \, d \alpha\right) \\
& = e^{\delta \langle b \rangle_{\widehat{I}}} \left(\mu(I)+    \int_{1}^{\infty} \mu(\{x \in I: |b(x)- \langle b \rangle_{\widehat{I}}|>\frac{\log \alpha}{\delta} \}) \, d \alpha  \right) \\
& \lesssim e^{\delta \langle b \rangle_{\widehat{I}}} \mu(I) \left( 1+ \int_{1}^{\infty} e^{- \frac{\log \alpha}{\delta} \delta_0} \, d\alpha  \right)\\
& \lesssim e^{\delta \langle b \rangle_{\widehat{I}}} \mu(I),
\end{align*}
as long as $\frac{\delta_0}{\delta} > 1$ (which is guaranteed by our hypothesis on $\delta$). Moreover, the implied constant is independent of $\delta$ and we conclude
$$ \langle e^{\delta b(x)} \rangle_{I} \lesssim e^{\delta \langle b \rangle_{\widehat{I}}}.  $$
Entirely similarly, we can also obtain the bounds 

$$ \langle e^{\delta b(x)} \rangle_{\widehat{I}} \lesssim e^{\delta \langle b \rangle_{\widehat{I}}}, \quad \langle e^{-\delta b(x)/(p-1)} \rangle_{I}^{p-1} \lesssim e^{-\delta \langle b \rangle_{\widehat{I}}}, \quad \langle e^{-\delta b(x)/(p-1)} \rangle_{\widehat{I}}^{p-1} \lesssim e^{-\delta \langle b \rangle_{\widehat{I}}}$$
with constants independent of $\delta.$ These estimates together prove the desired estimates for $[e^{\delta b}]_{\widehat{A}_p(\mu)}.$
\end{proof}

We can obtain a partial converse of Lemma \ref{sufficient} for general $A_2(\mu)$ weights, and a full converse in the case $\omega$ belongs to the stronger class $\widehat{A}_2(\mu).$

\begin{lemma} \label{PartialConverse}
If $\omega \in A_2(\mu)$, then $\log \omega$ satisfies

$$ \sup_{I \in \D} \langle |\log \omega- \langle \log \omega \rangle_{I}| \rangle_{I}< \infty .$$
In particular, the function $K_{\log \omega}^2 < \infty$ (as defined in (\ref{bmodef})), and if instead we assume the stronger condition $\omega \in \widehat{A}_2(\mu)$, then $\log \omega \in \BMO(\mu).$

\end{lemma}

\begin{proof}
For the first part, the classical proof works here without modification. Indeed, by Jensen's inequality 

\begin{align*}
\text{exp}\left( \langle |\log \omega- \langle \log \omega \rangle_{I}| \rangle_{I} \right) & = \text{exp}\left( \langle |\log \omega+ \langle \log \omega^{-1} \rangle_{I}| \rangle_{I} \right)\\
& \leq \left \langle \text{exp} \left( |\log \omega + \langle \log \omega^{-1} \rangle_{I}| \right) \right \rangle_{I}\\
& \leq \left \langle \text{exp} \left( \log \omega + \langle \log \omega^{-1} \rangle_{I} \right) \right \rangle_{I} + \left \langle \text{exp} \left( -\log \omega - \langle \log \omega^{-1} \rangle_{I} \right) \right \rangle_{I} \\
& = \langle \omega \rangle_{I} \cdot \text{exp}\left( \langle \log \omega^{-1} \rangle_{I} \right)+ \langle \omega^{-1} \rangle_{I} \cdot \text{exp}\left( \langle \log \omega \rangle_{I} \right)\\
& \leq 2 \langle \omega \rangle_{I} \langle \omega^{-1} \rangle_{I}\\
& \leq 2[\omega]_{A_2(\mu)}.
\end{align*}
If we instead assume $\omega \in \widehat{A}_2(\mu)$, the second assertion directly follows from the above computation, replacing $\langle \log \omega \rangle_{I}$ by $\langle \log \omega \rangle_{\widehat{I}}$ in the first expression.
\end{proof}

Theorem \ref{th:theoremB}, part \textbf{(a)} is a direct corollary of Lemmas \ref{sufficient} and \ref{PartialConverse}. We next focus on part \textbf{(b)}. The following straightforward property turns out to be crucial, because it lets us get around the lack of a doubling property for $\mu$ in stopping time arguments. If $\omega \in \widehat{A}_p(\mu)$, $\lambda>0$ and there exists $I \in\D$ maximal such that $\langle \omega \rangle_{I}> \lambda$, then
\begin{equation} 
\langle \omega \rangle_{I} \leq [\omega]_{\widehat{A}_p(\mu)} \lambda. \label{StoppingTime} 
\end{equation}
Indeed, using the $\widehat{A}_p(\mu)$ condition, H\"{o}lder's inequality, and maximality we have
$$\langle \omega \rangle_{I}  \leq \frac{[\omega]_{\widehat{A}_p(\mu)}}{\langle \sigma \rangle_{\widehat{I}}^{p-1}} \leq [\omega]_{\widehat{A}_p(\mu)} \langle \omega \rangle_{\widehat{I}} \leq [\omega]_{\widehat{A}_p(\mu)} \lambda.
$$

\begin{proof}[Proof of Theorem \ref{th:theoremB}, part \textbf{(b)}] We follow an argument which is similar to the one presented in \cite[Lemma 2.2]{HytonenPerezRela} and will not concern ourselves with sharp constants here. It is straightforward to check that if $\omega \in A_p(\mu)\supset \Ahat_p(\mu)$ for some $1<p<\infty$, then

\begin{equation}
\sup_{I \in \D} \frac{ (\MaximalD(\omega \mathbbm{1}_I))(I)}{\omega(I)}=\sup_{I \in \D} \frac{\int_{I} \MaximalD(\omega \mathbbm{1}_I) \, d\mu}{\omega(I)}=: [\omega]_{A_\infty(\mu)}< \infty.\label{Ainf}
\end{equation}
This property of $A_p(\mu)$ weights does not rely on $\mu$ being dyadically doubling, only the boundedness of the (unweighted) dyadic maximal function $\MaximalD$ on $L^{p'}(\sigma)$ for $\sigma \in A_{p'}(\mu).$ Indeed, assuming such a weighted bound for $\MaximalD$, we have
\begin{align*}
\int_{I} \MaximalD(\omega \mathbbm{1}_{I}) \, d\mu & \leq \left(\int_{I} \MaximalD(\omega \mathbbm{1}_{I})^{p'} \sigma  \, d\mu \right)^{1/p'} \omega(I)^{1/p} \\
&  \lesssim_{\omega} \left( \int_{I} \omega^{p'} \sigma \, d\mu \right)^{1/p'} \omega(I)^{1/p} = \omega(I),
\end{align*}
which proves \eqref{Ainf}. The weighted bound for the maximal function $\MaximalD:L^{q}(\sigma) \rightarrow L^{q}(\sigma)$ for $q=p'$ can be seen by observing the pointwise bound
$$ \MaximalD f(x_0) \leq [\sigma]_{A_q}^{1/(q-1)} \left\{\MaximalD^\sigma(\sigma^{-1} (\MaximalD^{\omega} (f \omega^{-1}))^{q-1})(x_0) \right \}^{1/(q-1)}, $$
and then applying the bounds for the weighted maximal functions $\MaximalD^\sigma$ and $\MaximalD^\omega$ on $L^{q'}(\sigma)$ and $L^{q}(\omega)$, respectively. Next, using the trivial bound $\omega(x) \leq \MaximalD(\omega \mathbbm{1}_I)(x)$ for $\mu$-almost every $x \in I$, we can reduce our goal to showing 
\begin{equation}
\left( \frac{1}{\mu(I)}\int_{I} [\MaximalD(\omega \mathbbm{1}_I)]^\gamma \, d\mu \right)^{1/\gamma} \lesssim_{\omega} \left( \frac{1}{\mu(I)}\int_{I} \omega\, d\mu \right)
\label{ReduRH}
\end{equation}
for some $\gamma>1$. Set $E_\lambda= \{x \in I: \MaximalD(\omega \mathbbm{1}_I)(x)> \lambda\} $. We write, using the distribution function and \eqref{Ainf},
\begin{align*}
\int_{I} [\MaximalD(\omega \mathbbm{1}_I)]^\gamma \, d\mu & = \int_{0}^{\infty} (\gamma-1) \lambda^{\gamma-2}\MaximalD(\omega \mathbbm{1}_I)\left(E_\lambda  \right) \, d\lambda \\
& \leq \int_{0}^{\langle \omega \rangle_{I}} (\gamma-1) \lambda^{\gamma-2}\MaximalD(\omega \mathbbm{1}_I)\left( I \right) \, d\lambda+ \int_{\langle \omega \rangle_{I}}^{\infty} (\gamma-1) \lambda^{\gamma-2}\MaximalD(\omega \mathbbm{1}_I)\left(E_\lambda  \right) \, d\lambda \\
& \leq [\omega]_{A_\infty(\mu)} \omega(I) \langle \omega \rangle_{I}^{\gamma-1} + (\gamma-1)\int_{\langle \omega \rangle_{I}}^{\infty} \lambda^{\gamma-2}\MaximalD(\omega \mathbbm{1}_I)\left(E_\lambda  \right) \, d\lambda.
\end{align*}
Let $J_j^\lambda$ be the collection of maximal intervals so that $\langle \omega \rangle_{J_j^\lambda}>\lambda$, and when $\lambda> \langle \omega \rangle_{I}$, these $J_j^\lambda$ are contained in $I$ and cover $E_\lambda.$ Crucially using \eqref{StoppingTime}, we have \begin{equation}\omega(J_j^\lambda) \leq [\omega]_{\widehat{A}_p(\mu)} \lambda \mu(J_j^\lambda)\label{ReplaceStopTime}.\end{equation} Then, we estimate the remaining integral on the right hand side as follows, using estimate \eqref{Ainf} again as well as \eqref{ReplaceStopTime}:

\begin{align*}
(\gamma-1)\int_{\langle \omega \rangle_{I}}^{\infty} \lambda^{\gamma-2}\MaximalD(\omega \mathbbm{1}_I)\left(E_\lambda  \right) \, d\lambda & = (\gamma-1) \int_{\langle \omega \rangle_{I}}^{\infty} \lambda^{\gamma-2} \sum_j \MaximalD(\omega \mathbbm{1}_I)\left(J_j^\lambda \right) \, d\lambda \\
& \leq [\omega]_{A_\infty(\mu)} (\gamma-1) \int_{\langle \omega \rangle_{I}}^{\infty} \lambda^{\gamma-2} \sum_j\omega(J_j^\lambda) \, d\lambda \\
& \leq [\omega]_{A_\infty(\mu)} [\omega]_{\widehat{A}_p(\mu)} (\gamma-1) \int_{\langle \omega \rangle_{I}}^{\infty} \lambda^{\gamma-1} \sum_j \mu (J_j^\lambda) \, d\lambda \\
& = [\omega]_{A_\infty(\mu)} [\omega]_{\widehat{A}_p(\mu)} (\gamma-1) \int_{\langle \omega \rangle_{I}}^{\infty} \lambda^{\gamma-1} \mu(E_\lambda) \, d\lambda \\
& \leq [\omega]_{A_\infty(\mu)} [\omega]_{\widehat{A}_p(\mu)} \left( \frac{\gamma-1}{\gamma} \right) \int_{I} [\MaximalD(\omega \mathbbm{1}_I)]^\gamma \, d\mu.
\end{align*}
Choosing $\gamma$ sufficiently close to $1$, we can guarantee that 
$$  [\omega]_{A_\infty(\mu)} [\omega]_{\widehat{A}_p(\mu)} \left( \frac{\gamma-1}{\gamma} \right)< \frac{1}{2},$$ which then implies

$$  \int_{I} [\MaximalD(\omega \mathbbm{1}_I)]^\gamma \, d\mu \leq 2 [\omega]_{A_\infty(\mu)} \omega(I) \langle \omega \rangle_{I}^{\gamma-1} ,$$
completing the proof.
    
\end{proof}

\subsection{Weighted estimates for \texorpdfstring{$[\Hd,b]$}{Lg}} We are ready to prove Theorem \ref{th:theoremC}. In this subsection, we go back to assuming that $\mu$ is sibling balanced and also assume $\mu$ is atomless for technical reasons. 

\begin{proof}[Proof of Theorem \ref{th:theoremC}]
 We assume $p=2$, but the case $p \neq 2$ can be handled in a similar manner with only minor modifications. The essential idea is to apply the Cauchy Integral trick, which is used in \cite{ChungPereyraPerez} for weighted estimates and was first used in \cite{CoifmanRochbergWeiss}. Consider the analytic family of operators
$$
\Hd_z f= e^{zb} \Hd(e^{-zb}f).
$$ 
If $|z|$ is sufficiently small, Lemma \ref{sufficient} shows $e^{|z|b}$ belongs to $\widehat{A}_2(\mu)$, and hence $\Hd_z$ extends to a bounded operator on $L^2(\mu)$ (see Appendix \ref{sec:app}). We wish to show now that $[\Hd, b]$ extends to a bounded operator on $L^2(\omega d\mu).$ For sufficiently small $\varepsilon>0$ depending on $\|b\|_{\BMO(\mu)}$, and bounded, compactly supported $f$, the function $\Hd_z(f)$ is analytic in the variable $z$, and thus the Cauchy integral formula gives
$$[b,\Hd]f= \frac{d}{dz}\biggr\lvert_{z=0} \left(\Hd_z (f)\right)= \frac{1}{2 \pi i}\int_{|z|=\varepsilon} \frac{\Hd_z(f)}{z^2} \,  dz.$$

We will show that we can choose $\varepsilon>0$ to be equal to $\frac{C(\omega,\mu)}{\|b\|_{\BMO(\mu)}}$, where $C(\omega,\mu)$ depends on $[\omega]_{\Ahat_2(\mu)}$ and $[\mu]_{sib}.$  In particular, suppose that we can show for this choice of $\varepsilon$, there holds
\begin{equation}
\sup_{|z|=\varepsilon}\|\Hd\|_{L^2(\omega e^{2 \text{Re}(z)\, b}) \rightarrow L^2(\omega e^{2 \text{Re}(z)\, b})} \lesssim  C(\omega, \mu).  \label{UniformSupNorm}
\end{equation}
Then, by Minkowski's integral inequality, we can estimate
\begin{align*}
\|[b,\Hd]f\|_{L^2(\omega)} & \leq \frac{1}{2 \pi} \int_{|z|=\varepsilon} \frac{\| \Hd_{z} f \|_{L^2(\omega)}}{\varepsilon^2}\, |dz|\\
& \lesssim \frac{1}{\varepsilon} \sup_{|z|=\varepsilon}\|\Hd_{z}\|_{L^2(\omega) \rightarrow L^2(\omega)} \|f\|_{L^2(\omega)}  \\
& \lesssim \frac{1}{\varepsilon}\sup_{|z|=\varepsilon}\|\Hd\|_{L^2(\omega e^{2 \text{Re}(z) b}) \rightarrow L^2(\omega e^{2 \text{Re}(z) b})}\|f\|_{L^2(\omega)} \\ 
& \lesssim C(\omega, \mu)\|b\|_{\BMO(\mu)} \|f\|_{L^2(\omega)},
\end{align*}
completing the proof that $[\Hd, b]$ is a bounded operator on $L^2(\omega).$

Again, to prove \eqref{UniformSupNorm} it is sufficient to observe that for this choice of $\varepsilon$ and $|z|=\varepsilon$, $\omega e^{2 \text{Re}(z) \, b}\in\widehat{A}_2(\mu)$ with weight characteristic uniformly bounded (in $z$) by some $C(\omega, \mu).$ Since $\omega, \omega^{-1} \in \widehat{A}_2(\mu)$ by hypothesis, both satisfy a reverse H\"{o}lder inequality by Theorem \ref{th:theoremB}. Thus, there exists $\gamma>1$ depending on $[\omega]_{\widehat{A}_2(\mu)}$ so that for all $I \in \D$

\begin{equation}
\left( \frac{1}{\mu(I)}\int_{I} \omega^\gamma \, d\mu \right)^{1/\gamma} \lesssim_{\omega} \left( \frac{1}{\mu(I)}\int_{I} \omega\, d\mu \right), \quad \left( \frac{1}{\mu(I)}\int_{I} \omega^{-\gamma} \, d\mu \right)^{1/\gamma} \lesssim_{\omega}  \left( \frac{1}{\mu(I)}\int_{I} \omega^{-1} \, d\mu \right).
\label{ReverseHolder}    \end{equation}

Now, choose $\varepsilon_0>0$ sufficiently small so that for any $0<\varepsilon \leq \varepsilon_0$, the exponent $2 \varepsilon \gamma'$ satisfies the conclusion of Lemma \ref{sufficient}. Note that $\varepsilon_0 \sim_{\omega} \|b\|_{\BMO(\mu)}^{-1}$. Therefore, if
$|z|=\varepsilon \leq \varepsilon_0$, we estimate, applying H\"{o}lder's inequality with exponents $\gamma,\gamma'$ and $J \in \{I, \widehat{I}\}$:
\begin{align*}
&  \left(\frac{1}{\mu(I)}\int_{I} \omega e^{2 \text{Re}(z)\, b} \, d\mu \right) \left( \frac{1}{\mu(J)}\int_{J} \omega^{-1} e^{-2 \text{Re}(z) \,b} \, d\mu \right)\\
& \leq  \left( \frac{1}{\mu(I)}\int_{I} \omega^\gamma  \, d\mu \right)^{1/\gamma} \left( \frac{1}{\mu(J)}\int_{J} \omega^{-\gamma}  \, d\mu \right)^{1/\gamma}\\
& \times \left( \frac{1}{\mu(I)}\int_{I} e^{2 \gamma' \text{Re}(z)\ b} \, d\mu \right)^{1/ \gamma'} \left( \frac{1}{\mu(J)}\int_{J} e^{-2 \gamma' \text{Re}(z)\ b} \, d\mu \right)^{1/ \gamma'} \\
& \lesssim C(\omega, \mu)  \left( \frac{1}{\mu(I)}\int_{I} \omega  \, d\mu \right)\left( \frac{1}{\mu(J)}\int_{J} \omega^{-1}  \, d\mu \right)\\
& \times \left( \frac{1}{\mu(I)}\int_{I} e^{2 \gamma' \text{Re}(z)\ b} \, d\mu \right)^{1/ \gamma'} \left( \frac{1}{\mu(J)}\int_{J} e^{-2 \gamma' \text{Re}(z)\ b} \, d\mu \right)^{1/ \gamma'} \\
& \lesssim C(\omega, \mu) [e^{2 \gamma' \text{Re}(z) \,b}]_{\widehat{A}_2(\mu)}^{1/\gamma'}\\
& \lesssim C(\omega,\mu).
\end{align*}
Obviously, the same estimates hold with the roles of $I$ and $J$ interchanged. We have shown $\omega e^{2 \text{Re}(z)\, b}\in\widehat{A}_2(\mu)$ with 
$$ \sup_{|z|=\varepsilon} [\omega e^{2  \text{Re}(z) \, b}]_{\widehat{A}_2(\mu)} \lesssim C(\omega, \mu).$$ 
By Proposition \ref{AvgFather} one can also bound $[\omega e^{2\text{Re}(z)b}]_{A_{2}^{sib}(\mu)}$ uniformly in $|z|=\varepsilon$
which combined with Corollary \ref{weightedresultforH} implies \eqref{UniformSupNorm}. That concludes the proof by the inequalities immediately following \eqref{UniformSupNorm}. 

\end{proof}

\begin{rem}
In the special case $\omega \equiv 1$, Theorem \ref{th:theoremC} provides an alternative proof of the upper bound for $[\Hd,b]$ that does not use either of the paraproduct decompositions given in Section \ref{sec1}.
\end{rem}

The following analog of Theorem \ref{th:theoremC} holds for general Haar shifts $T$ like those introduced in Section \ref{subsec14}. The proof is very similar, and we leave the details to the interested reader. 

\begin{prop} \label{CommutatorWeighted}
Let $\mu$ be balanced and atomless, $1<p<\infty$, $T$ be a Haar shift, $b \in \BMO(\mu)$, and $\omega \in \widehat{A}_p(\mu)$. Then there holds 
\begin{equation*}\|[T,b]\|_{L^p(\omega)\rightarrow L^{p}(\omega)} \lesssim_{[\omega]_{\Ahat_p}, p,\mu} \|b\|_{\BMO(\mu)}.
\end{equation*}
\end{prop}

\subsection{Comparison between different weight classes}\label{subseccompareweights}

\begin{prop}If $\mu$ is balanced, there holds
\begin{equation}\label{SibBalEquiv}
A_p^{sib}(\mu)=A_p^{bal}(\mu).
\end{equation}
\begin{proof}
First, notice for the configurations $I=J$, $I=(\widehat{J})^s$ or $J= (\widehat{I})^s$, the balanced condition gives $c_p(I,J) \sim \frac{m(I)^{p/2} m(J)^{p/2}}{\mu(I)^{p-1} \mu (J) }$, with constant depending on $[\mu]_{bal}.$ This observation proves $A_p^{sib}(\mu) \subseteq A_p^{bal}(\mu).$ To show the reverse containment, it remains to check the additional configuration where $\widehat{I}= \widehat{J}^s$ that arises in the $A_p^{sib}$ condition. We can bound this term via a case analysis. If $\mu(\hat{I})\leq 2 \mu(I)$, then we can estimate, since $m(\hat{J}) \sim m(J)$ as well:

\begin{align*}
c_p(I, J) \langle \omega \rangle_I  \langle \sigma \rangle_{J}^{p-1} & \lesssim \frac{m(\hat{I})^{p/2} m(J)^{p/2}}{\mu(\hat{I})^{p-1} \mu(J)} \langle \omega \rangle_{\hat{I}} \langle \sigma \rangle_{J}^{p-1}\\
& \leq [\omega]_{A_p^{bal}(\mu)}.
\end{align*}

If $\mu(\hat{J}) \leq 2 \mu(J)$ instead, we follow a similar argument. If neither of these inequalities hold, we have $\mu(I) \sim m(\hat{I}) \sim m(\hat{J}) \sim \mu(J)$. Let $I^g$ denote the dyadic parent of $\widehat{I}.$ Then since $m(\hat{I})$ is comparable to $m(I^g)$, which is in turn comparable to $\min \{ \mu(\hat{I}), \mu(\hat{J})\},$ this implies either $\mu(\hat{I}) \sim \mu(I)$ or $\mu(\hat{J}) \sim \mu(J) $, with implicit constant only depending on $[\mu]_{bal}.$ Then we can argue as above. 
\end{proof}
\end{prop}

\begin{prop} \label{AvgFather}

Let $\mu$ be sibling balanced. The containment $$\widehat{A}_p(\mu) \subseteq A_p^{sib}(\mu)$$ holds, and in particular for $\omega \in \widehat{A}_p(\mu)$,
$$[\omega]_{A_p^{sib}(\mu)} \lesssim [\omega]_{\widehat{A}_p(\mu)}^3, $$
where the implicit constant is independent of $\omega, \mu.$
Therefore, if $\mu$ satisfies the stronger balanced condition, we get the corresponding result for $A_p^{bal}(\mu)$ in place of $A_p^{sib}(\mu).$

\begin{proof}
Let $w \in \widehat{A}_p(\mu).$ Since $c_p(I,J) \lesssim 1$ always, to verify \eqref{A2b} it suffices to show 
\begin{equation} \sup_{\substack{I, J \in \D:\\ J \in \{I, I_{-}^s, I_{+}^s\} \text{ or }\\ I \in \{J,J_{-}^s, J_{+}^s\} \text{ or } \\ \widehat{I}=\widehat{J}^s}}  \langle \omega \rangle_{I} \langle \sigma \rangle_{J}^{p-1}< \infty.\end{equation}
We will only verify the case corresponding to $\widehat{I}=\widehat{J}^s$; the other cases can be handled by a similar argument.  We have, using H\"{o}lder's inequality:

\begin{align*}
\langle \omega \rangle_{I} \langle \sigma \rangle_{J}^{p-1} & \leq \langle \omega \rangle_{I} \langle \sigma \rangle_{\widehat{I}}^{p-1} \langle \omega \rangle_{\widehat{I}} \langle \sigma \rangle_{J}^{p-1}\\
& \leq [\omega]_{\widehat{A}_p(\mu)} \langle \omega \rangle_{\widehat{I} } \langle \sigma \rangle_{J}^{p-1}\\
& \leq [\omega]_{\widehat{A}_p(\mu)} \langle \omega \rangle_{\widehat{I} } \langle \sigma \rangle_{\widehat{I}^s}^{p-1} \langle \omega \rangle_{\widehat{I}^s}  \langle \sigma \rangle_{J}^{p-1}\\
& \leq  [\omega]_{\widehat{A}_p(\mu)}^2 \langle \omega \rangle_{\widehat{I} } \langle \sigma \rangle_{\widehat{I}^s}^{p-1}.
\end{align*}

Now, if $I^{g}$ denotes the dyadic parent of $\widehat{I}$, we know that either $\mu(I^g) \leq 2 \mu(\widehat{I})$ or $\mu(I^g) \leq 2 \mu(\widehat{I}^s).$ In the former case, we have $\langle \omega \rangle_{\widehat{I}} \leq 2 \langle \omega \rangle_{I^g}$, while in the latter, we get $\langle \omega \rangle_{\widehat{I}^s}^{p-1} \lesssim 2^{p-1} \langle \omega \rangle_{I^g}^{p-1} .$ In either case, the last line in the display above is controlled by $2^{\max\{1, p-1\}} [\omega]_{\widehat{A}_p(\mu)}^3$, which completes the proof. 
\end{proof}
    
\end{prop}

The next result is substantially more delicate, and demonstrates that for general balanced $\mu$, the containment $\Ahat_2(\mu) \subset A_2^{bal}(\mu)$ may be strict (even if $\mu$ is assumed to be atomless). Therefore, $\Ahat_2(\mu)$ defines a genuinely different class of weights in the category of atomless balanced measures.

\begin{thm} \label{counterexample}
There exists a balanced, atomless measure $\mu$ and weight $\omega \in A_2^{bal}(\mu)$ such that $\omega \notin \Ahat_2(\mu).$ Consequently, we have the strict containment $\widehat{A}_2(\mu) \subsetneq A_2^{bal}(\mu)$ for this particular measure. 

\begin{proof}
Fix an even positive integer $n \geq 2$. We will construct a preliminary measure $\mu_n$ and weight function $\omega_n$ on $[0,1]$ so that $\mu_n$ is balanced with $[\mu_n]_{bal} \sim 1$, $[\omega_n]_{\Ahat_2} \gtrsim \sqrt{n}$, but $[\omega_n]_{A_2^{bal}(\mu_n)} \lesssim 1.$ Then, we will show how these sequences of measures and weight can be used to build $\mu, \omega$ satisfying the conclusion of the theorem. 

Set $I_k=[0, 2^{-k}),$ and let $I_k^s=[2^{-k}, 2^{-k+1})$ denote its dyadic sibling. We set $\mu_n(I_0)=1$, $\mu_n(I_1)=\mu_n(I_1^s)=\frac{1}{2}$, and define sequences
$$b_k= \begin{cases} 
\frac{1}{k^2}, &2 \leq k \leq n;\\
\frac{1}{n^2}, & k>n.
\end{cases}, \, \text{and}, \quad a_k=1-b_k.$$ For $k \geq 2$ we then inductively define  $\mu_n(I_k)= a_k \mu_n(I_{k-1})$, so that $\mu_n(I_k^s)=b_k \mu_n(I_{k-1}).$ This construction defines a Borel measure $\mu$ on $[0,1]$ which is uniform on each interval $I_k^s$ (see \cite{LSMP} or \cite{CAPW}).

On account of the convergence of the infinite product $\prod_{j=2}^{\infty}(1-\frac{1}{j^2})$ to a nonzero value, one can check that $\mu_n(I_k) \sim 1$ for all $k \leq n$ and $\mu_n(I_k^s) \sim \frac{1}{k^2} $ for the same range of $k$, with implied constant independent of $n$. Moreover, for $k \leq n$ we also have  $$ m(I_k) \sim m(I_k^s) \sim m(\widehat{I_k}) \sim \frac{1}{k^2}.$$ 

On the other hand, notice for $k>n$, $\mu_n(I_k) \sim (\frac{n^2-1}{n^2})^{k-n}$ while $\mu_n(I_k^s) \sim (\frac{1}{n^2}) (\frac{n^2-1}{n^2})^{(k-n-1)}$. In this case, 

$$ m(I_k) \sim m(I_k^s) \sim m(\widehat{I_k}) \sim \left(\frac{1}{n^2}\right) \left(\frac{n^2-1}{n^2}\right)^{(k-n-1)}.$$ 
We only need to check such intervals to verify the balanced condition because $\mu_n$ is uniform on $I_k^s$ for all $k.$ Moreover, since $\mu(I_k) \rightarrow 0$ as $k \rightarrow \infty$, $\mu$ is atomless. 

Next, we define the relevant weight function $\omega_n$ on $[0,1].$ Define $$\omega_n(x)= \begin{cases} \frac{1}{\sqrt{k}}, & x \in I_k^s, k \leq n/2 \\ \frac{1}{\sqrt{n-k+1}}, & x \in I_k^s, n/2<k\leq n \\ 1,  & x \in I_k^s, k>n. \end{cases}$$
We first claim that $[\omega_n]_{\Ahat_2} \gtrsim \sqrt{n}.$ To see this, note that
\begin{align*}
\langle \omega_n^{-1} \rangle_{I_{n/2}^s} & = \sqrt{n/2},
\end{align*}
while on the other hand, we have
$$\langle \omega_n \rangle_{\widehat{I_{n/2}}} \gtrsim 1,$$ since $\mu_n(\widehat{I_{n/2}}) \sim 1$ and $\int_{\widehat{I_{n/2}}} \omega_n \, d \mu_n \geq \mu_n(I_n) \sim 1.$
This proves the desired estimate on $[\omega_n]_{\Ahat_2}$.

Next, we assert $\omega_n \in A_2(\mu_n)$ with $[\omega_n]_{A_2(\mu_n)} \sim 1.$ It is obvious that since $\omega_n$ is constant on $I_k^s$ for all $k,$ and $I_{k}$ with $k>n$, we only need to check the $A_2$ condition on intervals of the form $I_k$ with $k \leq n.$ As before, $\langle \omega_n \rangle_{I_k} \leq 1,$ and we also see
\begin{align*}
\langle \omega_n^{-1} \rangle_{I_k} & \lesssim \int_{I_k} \omega_n^{-1} \, d \mu_n\\
& = \sum_{\substack{j > k:\\j \leq n/2}} \sqrt{j} \mu_n(I_j^s) + \sum_{\substack{j > k:\\n/2<j\leq n}} \sqrt{n-j+1} \mu_n(I_j^s) + \sum_{j > n} \mu_n(I_j^s) \\
& \lesssim  \sum_{j \geq 1 } \frac{1}{j^{3/2}} + \frac{n \sqrt{n}}{n^2}+ \sum_{j \geq 1 } \frac{1}{n^2} \left(\frac{n^2-1}{n^2}\right)^{j} \lesssim 1,
\end{align*}
so $\omega_n \in A_2(\mu_n). $    

Finally, we need to check the additional conditions for membership in $A_2^{bal}(\mu_n).$ We need to control 
$$ \frac{m(I) m(J)}{\mu_n(I)\mu_n(J) } \langle \omega_n \rangle_{I} \langle \omega_n^{-1} \rangle_{J}   $$ for the combinations
$I=I_k^s, J=I_{k+1}^s, \quad I= I_k^s, J=I_{k+1}, \quad I=I_k, J= I_{k,-}^s,$ and  $\quad I=I_k, J= I_{k,+}^s$ as well as the combinations with $I$ and $J$ interchanged in this list. For all of the combinations, we can assume $k \leq n$, for otherwise $\omega_n$ is identically constant on the two intervals. 

Note that $\langle \omega_n \rangle_{I} \leq 1$ for any dyadic interval $I$, and by our earlier computations, $\langle \omega_n^{-1} \rangle_{J} \lesssim 1 $ unless $J=I_k^s$ with $1 \leq k \leq n$ or $J$ is a child of $I_k^s$. Then we estimate for the case $J=I_k^s$ and $I=I_{k+1}^s$:
\begin{align*}
\frac{m(I_k^s) m(I_{k+1}^s)}{\mu_n(I_{k}^s) \mu_n (I_{k+1}^s) } \langle \omega_n\rangle_{I_{k+1}^s} \langle \omega_n^{-1} \rangle_{I_{k}^s} & \lesssim \langle \omega_n\rangle_{I_{k+1}^s} \langle \omega_n^{-1} \rangle_{I_{k}^s} \\ & \lesssim 1,
\end{align*}
while in the case $I=I_{k+1}$ instead,
\begin{align*}
\frac{m(I_k^s) m(I_{k+1})}{\mu_n(I_{k}^s) \mu_n (I_{k+1}) } \langle \omega_n\rangle_{I_{k+1}} \langle \omega_n^{-1} \rangle_{I_{k}^s} & \lesssim \frac{ m(I_{k+1})}{\mu_n (I_{k+1}) } \langle \omega_n^{-1} \rangle_{I_{k}^s} \\ & \lesssim \frac{1}{k^2} \cdot \sqrt{k} \lesssim 1.
\end{align*}
The cases when $J$ is a child of $I_k^s$ may be handled similarly. These estimates prove $\omega_n \in A_2^{bal}(\mu_n)$ with $[\omega_n]_{A_2^{bal}} \sim 1$, which completes the proof of the claim.

Now, to construct a single measure $\mu$ and weight function $\omega$ satisfying the conclusion of the lemma, we define $\mu$ on the half line $[0, \infty)$ using the $\mu_n$. In particular, let $u_n$ be the density function for $\mu_{2(n+1)}$ (which is absolutely continuous with respect to Lebesgue measure) as constructed above for $n \geq 1$ on $[0,1]$, and define $\mu$ according to the density

$$ d \mu(x)= u(x) \, dx, \quad u(x)= u_0(x)+ \sum_{n=1}^{\infty} f_n(x),$$
where $u_0(x)$ is the positive density function on $[0,1]$ of the measure $\mu_2$, and for $n \geq 1$, $f_n$ is the positive function supported on $[2^{n-1}, 2^{n})$ defined by 
$$ f_n(x)= u_n\left(\frac{x-2^{n-1}}{2^{n-1}}\right).$$ It is a simple matter to check that $\mu$ is balanced, since each $\mu_n$ is balanced with $[\mu_n]_{bal} \sim 1.$  Next, we build $\omega$ from the weights $\omega_n$ appropriately. In particular, we can define 

$$ \omega(x)= \omega_2(x)+ \sum_{n=1}^{\infty} \nu_n(x),$$
where $$ \nu_n(x)= \omega_{2(n+1)}\left(\frac{x-2^{n-1}}{2^{n-1}}\right), n\geq 1.$$
From the computations above, it is easy to see $\omega \notin \Ahat_2(\mu)$, but $\omega \in A_2^{bal}(\mu).$ This completes the construction.

\end{proof}
\end{thm}

\appendix
\section{Sparse bounds and weighted estimates for  \texorpdfstring{$\Hd$}{Lg}} \label{sec:app}

If $\mu$ is sibling balanced, then as discussed in Proposition \ref{weak11dyadicH}, $\Hd$ and $\Hd^*$ are of weak-type $(1,1)$, and hence $\Hd$ is bounded on $L^p(\mu)$, for all $1<p<\infty$. Moreover, that condition on the measure is necessary for $\mathcal{H}$ to be bounded in some $L^{p}(\mu)$ with $p\neq 2$. Then, in the context of atomless, sibling measures, we shall show that $\Hd$ satisfies a sparse bound in the spirit of the one proved in \cite[Lemma 2.4]{CAPW}. We will also state the weighted inequalities that descend from the sparse domination, which in particular imply the unweighted $L^p(\mu)$-bounds for $\mathcal{H}$.

\subsection{Sparse bounds for \texorpdfstring{$\Hd$}{Lg}}
We follow \cite{CAPW} rather closely here, so we will skip some details. Let $0\leq f_1,f_2\in L^{1}(\mu)$ supported in $I\in  \mathcal{\D}$. Using \cite[Lemma 2.2]{CAPW}, we can perform a Calder\'on-Zygmund decomposition at two different heights $\lambda_1,\lambda_2$ simultaneously for $f_1,f_2$, so that we find $g_j,b_j$ for $j=1,2$ with the following properties:
     \begin{enumerate}
         \item\label{CZbadpieces} There exist pairwise disjoint dyadic intervals $\{I_k\}\subset \D(I)$ so that for each $j=1,2$, $b_j=\sum_k b_{j,k}$ where $\text{supp}(b_{j,k})\subset \hat{I}_k$, $\int_{\hat{I}_{j,k}} b_{j,k}=0$ and $\|b_{j,k}\|_{L^{1}(\mu)}\lesssim \int_{I_k} f_j d\mu$. Moreover,
    $$b_{j,k}=f_j\mathbbm{1}_{I_k}-\langle f_j \mathbbm{1}_{I_k}\rangle_{\hat{I}_k}\mathbbm{1}_{\hat{I}_k}.$$
    \item\label{CZLpcontrol} For $j=1,2$, one has $g_j\in L^{p}(\mu)$ for all $1\leq p<\infty$ with $\|g_j\|_{L^{p}(\mu)}^p\leq C_p \lambda_j^{p-1}\|f_j\|_{L^{1}(\mu)}.$
    \item\label{CZBMOcontrol} For $j=1,2$, one has $g_j\in BMO(\mu)$ and $\|g_j\|_{\BMO(\mu)}\leq \lambda_j$.
     \end{enumerate}
   Fix an interval $I_0\in \D$. For any $I\subset I_0$, we can apply the Calderón Zygmund decomposition above to the functions $f_1\one_I$ and $f_{2}\one_I$ at heights $\lambda_i=16 \langle f_{i}\rangle_I$ to get a collection of pairwise disjoint intervals $\mathcal{B}(I)\subset \D(I)$. We set
\begin{align*}
    \mathcal{B}_1(I)=\cup_{J\in \mathcal{B}(I)}\mathcal{B}(J),\, \mathcal{B}^{1}(I)=\mathcal{B}(I)\cup \mathcal{B}_1(I), \\
    \mathcal{G}(I)=\{J\in \D(I)\colon J\not\subset K \text{ for any }K\in \mathcal{B}(I)\}.
\end{align*}
For each $J\in \mathcal{G}(I)$ we have 
$\langle f_i\rangle_{K}\leq \lambda_i\,\text{ for all }J\subseteq K\subseteq I,\,i=1,2$.

\begin{lemma}[Iteration step] \label{lem:iterStep}
We have  
\begin{align*}
| \sum_{J\in \mathcal{G}(I)} &\fonehaarJplus \ftwohaarJminus-\fonehaarJminus\ftwohaarJplus| \lesssim  [\mu]_{sib}^{1/2}\avfoneI \avftwoI \mu(I) \\
& +\sum_{\substack{S,T\in \mathcal{B}^{1}(I):\\  \hat{S}=(\hat{T})^s}}  \avfoneS \avftwoT m(\hat{S})^{1/2}{m(\hat{T})}^{1/2} + \sum_{\substack{S,T\in \mathcal{B}(I):\\  {S}=T^s}}  \avfoneS \avftwoT m({S})^{1/2}{m({T})}^{1/2}\\
& + \sum_{\substack{S\in \mathcal{B}(I),T\in \mathcal{B}_1(I) \\ \text{ or }S,T\in \mathcal{B}(I),\\ \text{ and } S=(\hat{T})^s}}  (\avfoneS \avftwoT+\langle f_1\rangle_T\langle f_2\rangle_{S}) m({S})^{1/2}{m(\hat{T})}^{1/2}.
\end{align*}

\end{lemma}

\begin{proof} One has
\begin{equation*}
\mathrm{LHS} \leq | \sum_{J\in \mathcal{G}(I)} \fonehaarJplus \ftwohaarJminus|+ | \sum_{J\in \mathcal{G}(I)} \fonehaarJminus\ftwohaarJplus|=:\mathrm{A}+\mathrm{B}.
\end{equation*}
Both terms are handled similarly so let us focus on A. We pick $\lambda_i=16\langle f_i\rangle_{I}$, $i=1,2$ in the Calder\'on-Zygmund decomposition above applied to $f_i\one_I$. Then 
\begin{equation*}
\begin{split}
    | \sum_{J\in \mathcal{G}(I)} &\fonehaarJplus \ftwohaarJminus| \leq  |\sum_{J\in \mathcal{G}(I)}\gonehaarJplus \gtwohaarJminus|+|\sum_{J\in \mathcal{G}(I)}\bonehaarJplus \gtwohaarJminus|\\
    &+|\sum_{J\in \mathcal{G}(I)}\gonehaarJplus \btwohaarJminus|+|\sum_{J\in \mathcal{G}(I)}\bonehaarJplus \btwohaarJminus|\\
    &=:\mathrm{(GG)}+\mathrm{(BG)}+\mathrm{(GB)}+\mathrm{(BB)}.
\end{split}
\end{equation*}
The $L^2$-bound for $g_i$ yields $\mathrm{(GG)} \lesssim \avfoneI\avftwoI \mu(I)$. Let us estimate (BG). First, $\bonejhaarJplus=0$ unless $J_{+}\in\{I_j,\hat{I_{j}}\}$. Indeed, $\text{supp}(b_j)=\hat{I}_j$ se we only need to consider $J$ so that $J_{+}\cap \hat{I}_j\neq \emptyset$, i.e. $J_{+}\subseteq \hat{I}_j$ or $\hat{I}_j\subsetneq J_{+}$. In the latter case, one gets zero because $h_{J_{+}}$ is constant in $\text{supp}(b_{1,j})$, which has integral zero. If $J_{+}\subseteq I_j^s$, we get zero for similar reasons. We are left with $J_{+}=\hat{I}_j$ or $J_{+}= I_j$, for $J_{+}\subsetneq I_j$ is impossible beacause $J\in \mathcal{G}(I)$. Hence
\begin{align*}
    \mathrm{(BG)} & \leq \sum_j \sum_{\substack{J\in \mathcal{G}(I):\\  J_{+} \in \{I_j,\hat{I}_j\}}} |\bonejhaarJplus||\gtwohaarJminus| \\
    & \lesssim \sum_j\sum_{\substack{J\in \mathcal{G}(I):\\  J_{+} \in \{I_j,\hat{I}_j\}}} \frac{1}{\sqrt{m(J_{+})}}\|b_{1,j}\|_{L^{1}(\mu)}\sqrt{m(J_{-})}\|g_2\|_{\BMO(\mu)} \\
    & \lesssim \sum_j\sum_{\substack{J\in \mathcal{G}(I):\\  J_{+} \in \{I_j,\hat{I}_j\}}} \frac{\sqrt{m(J_{-})}}{\sqrt{m(J_{+})}}\int_{I_j} f_1d\mu\,\avftwoI \\
    & \lesssim [\mu]_{sib}^{1/2}\sum_j  \int_{I_j} f_1d\mu \, \avftwoI\lesssim [\mu]_{sib}^{1/2}\avfoneI \avftwoI \mu(I),
\end{align*}
using the sibling balanced condition and the disjointness of $\{I_j\}$. The term (GB) is estimated similarly to (BG). To estimate (BB), the only intervals $J$ contributing to the sum are so that $J_{+}\in \{I_j,\hat{I}_j\}$ and $J_{-}\in \{I_k,\hat{I}_k\}$. Therefore, 
\begin{align*}
    \mathrm{(BB)} & \leq \sum_{j,k} \sum_{\substack{J\in \mathcal{G}(I):\\ J_{+}\in \{I_j,\hat{I}_j\}\\J_{-}\in \{I_k,\hat{I}_k\}}}|\bonejhaarJplus \btwokhaarJminus| \\
    & \leq \sum_{j,k\colon \hat{I_j}=(\hat{I}_k)^{s}}|\bonejhaarIjhat \btwokhaarIkhat |+  \sum_{j,k\colon I_j=(I_k)^s}|\bonejhaarIj \btwokhaarIk |\\
    & +\sum_{j,k\colon I_j=(\hat{I}_k)^s}|\bonejhaarIj \btwokhaarIkhat |+ \sum_{j,k\colon I_k=(\hat{I}_j)^s}|\bonejhaarIjhat \btwokhaarIk|\\
    &=:\mathrm{(BB)}_1+\mathrm{(BB)}_2+\mathrm{(BB)}_3+\mathrm{(BB)}_4.
\end{align*}
Using that $|\bonejhaarIjhat|=\sqrt{m(\hat{I}_j)}\langle f_1\rangle_{I_j}$ and a similar equality for $b_{2,k}$, we have 
$$
\mathrm{(BB)}_1\leq \sum_{\substack{S,T\in \mathcal{B}(I)\colon \\  \hat{S}=\hat{T}^s}} \langle f_1\rangle_{S}\langle f_2\rangle_{T} \sqrt{m(\hat{S})}\sqrt{m(\hat{T})}.
$$
For $\mathrm{(BB)}_2$ we use $\bonejhaarIj=\sqrt{m(I_j)} (\langle f_1\rangle_{(I_j)_{-}}-\langle f_1\rangle_{(I_j)_{+}})$, and so 
\begin{equation*}
    \begin{split}
        \mathrm{(BB)}_2 & \leq \sum_{j,k \colon I_j=(I_k)^s} \sqrt{m(I_j)}\sqrt{m(I_k)} |\langle f_1\rangle_{(I_j)_{-}}-\langle f_1\rangle_{(I_j)_{+}}|  |\langle f_2\rangle_{(I_k)_{-}}-\langle f_2\rangle_{(I_k)_{+}}|\\
        & \leq \sum_{j,k \colon I_j=(I_k)^s} \sqrt{m(I_j)}\sqrt{m(I_k)} \sum_{P\in \D_1(I_j),Q\in \D_1(I_k)} \langle f_1\rangle_{P}\langle f_2\rangle_{Q}.
    \end{split}
\end{equation*}
In the sum above, if $P\in \mathcal{B}(I_j)\subset \mathcal{B}_{1}(I)$ we keep it as it is. Otherwise, we estimate $\langle f_1\rangle_{P}\lesssim \langle f_1\rangle_{I_j}$, with $I_j\in \mathcal{B}(I)$. We argue similarly for each $Q$. That gives rise to four possibilities, which can all be bounded by the right hand side in the statement:
\begin{equation*}
  \begin{split}
     \mathrm{(BB)}_2 & \lesssim \sum_{S,T\in \mathcal{B}_{1}(I)\colon \hat{S}=(\hat{T})^s}\sqrt{m(\hat{S})}\sqrt{m(\hat{T})}\avfoneS \avftwoT + \sum_{\substack{S\in \mathcal{B}(I),T\in \mathcal{B}_1(I)\\ S=(\hat{T})^s} } \sqrt{m(S)}\sqrt{m(\hat{T})}\avfoneS\avftwoT \\ 
     & +\sum_{\substack{T\in \mathcal{B}(I),S\in \mathcal{B}_1(I)\\ T=(\hat{S})^s} } \sqrt{m({T})}\sqrt{m(\hat{S})}\avfoneS\avftwoT +\sum_{\substack{S,T\in \mathcal{B}(I)\\ {S}={T}^s} } \sqrt{m({S})}\sqrt{m({T})}\avfoneS\avftwoT.
  \end{split}  
\end{equation*}
Finally, $\mathrm{(BB)}_3$ and $\mathrm{(BB)}_4$ are estimated following similar ideas and we omit the details.
\end{proof}

A family of intervals $\mathcal{S}$ is said to be $\eta$-sparse if for each $I \in \mathcal{S}$ the exists $E_I \subset I$ such that $\mu(E_I) \geq \eta \mu(I)$ and the family $\{E_I\}$ is pairwise disjoint. Since $\mu$ is non-atomic, \cite[Theorem 1.3]{Hanninen} implies that $\mathcal{S}$ being $\eta$-sparse is equivalent to the Carleson packing condition
$$
\sum_{I\supsetneq J\in\mathcal{S}} \mu(J) \leq \frac{1}{\eta} \mu(I), \quad \mbox{for all }I\in \mathcal{S}. 
$$
If $\mathcal{S}$ is sparse, we define
\begin{equation*}
    \mathcal{A}_{\mathcal{S}}(f_1,f_2):=\sum_{I\in \mathcal{S}} \avfoneI \avftwoI \mu(I),
\end{equation*}
and the non-standard forms
\begin{equation*}
         \begin{split}
\mathcal{E}^{\mathcal{S}}_1(f_1,f_2):=&\sum_{\substack{I,J\in \mathcal{S}\\ \hat{I}=(\hat{J})^{s}}}\avfoneI \avftwoJ \,m(\hat{I})^{1/2} m(\hat{J})^{1/2},\\
\mathcal{E}^{\mathcal{S}}_2(f_1,f_2):=&\sum_{\substack{I,J\in \mathcal{S}\\ {I}=(\hat{J})^{s}}}\avfoneI \avftwoJ \,m(I)^{1/2} m(\hat{J})^{1/2},\\
\mathcal{E}^{\mathcal{S}}_3(f_1,f_2):=&\sum_{\substack{I,J\in \mathcal{S}\\ {J}=(\hat{I})^{s}}}\avfoneI \avftwoJ \,m(\hat{I})^{1/2} m({J})^{1/2}.
         \end{split}
    \end{equation*}

\begin{thm}[Sparse Domination for $\Hd$]\label{sparseresultforH}
    If $\mu$ is sibling balanced and atomless, there exists $\eta\in (0,1) $ such that for each pair $(f_1,f_2)$ of nonnegative compactly supported bounded functions there exists an $\eta$-sparse collection $\mathcal{S}\subset \D$ such that 
    $$|\langle\Hd f_1, f_2\rangle|\lesssim [\mu]_{sib}^{1/2}\mathcal{A}_{\mathcal{S}}(f_1,f_2)+\sum_{i=1}^{3}\mathcal{E}_{i}^{\mathcal{S}}(f_1,f_2).
    $$
\end{thm}

\begin{proof}
Let     
$$\mathcal{E}^{\mathcal{S}}_4(f_1,f_2):=\sum_{\substack{I,J\in \mathcal{S}\colon  I=J^{s}}}\avfoneI \avftwoJ\, m(I)^{1/2} m(J)^{1/2}.
$$
Following the proof scheme of \cite[Theorem 2.3]{CAPW}, Lemma \ref{lem:iterStep} implies that there exists $\eta_0\in (0,1)$ and an $\eta_0$-sparse family $\mathcal{S}_0$ such that
$$|\langle\Hd f_1, f_2\rangle|\lesssim [\mu]_{sib}^{1/2}\mathcal{A}_{\mathcal{S}_{0}}(f_1,f_2)+\sum_{i=1}^{4}\mathcal{E}_{i}^{\mathcal{S}_0}(f_1,f_2).$$
We claim that there exists a sparse $\mathcal{S}\supset \mathcal{S}_0$ so that
$$\mathcal{E}_4^{\mathcal{S}_{0}}(f_1,f_4)\leq C [\mu]_{sib}^{1/2}\mathcal{A}_{\mathcal{S}}(f_1,f_2).$$
Indeed, for each $I\in \mathcal{S}_{0}$ such that $I^{s}\in \mathcal{S}_0$, 
$$\avfoneI \langle f_2\rangle_{I^s}\, m(I)^{1/2} m(I^s)^{1/2}\leq \langle f_1\rangle_{\hat{I}} \langle f_2\rangle_{\hat{I}} \mu(\hat{I})\frac{m(I)^{1/2}m(I^s)^{1/2}\mu(\hat{I})}{\mu(I)\mu(I^s) }.$$
Without loss of generality assume that $\mu(I)\leq \mu(I^s)$. Then 
$$\frac{\mu(\hat{I})}{\mu(I^s)}\frac{m(I)^{1/2}m(I^s)^{1/2}}{\mu(I)}\leq 2\frac{[\mu]_{sib}^{1/2}m(I)}{\mu(I)}\leq 2[\mu]_{sib}^{1/2}.$$
Then, it is enough to define $\mathcal{S}$ by adding to $\mathcal{S}_0$ new intervals $\hat{I}$ whenever $I$ and $I^s$ are both in $\mathcal{S}_0$. Then $\mathcal{S}$ is $\eta$-sparse for some $0<\eta=\eta(\eta_0)<\eta_0$, which is enough.
\end{proof}

\subsection{Adapted weight classes} Let $1<p<\infty$. A weight $\omega \in A_p^{sib}(\mu)$ if 
    \begin{equation*}
        [\omega]_{A_{p}^{sib}(\mu)}:=\sup_{I.J\in \D}c_{p}(I,J)\langle \omega\rangle_{I} \langle \sigma\rangle_J^{p-1} <\infty,
    \end{equation*}
    where 
    \begin{equation*}\label{defconstants}
        c_p(I,J)=\begin{cases}
            1, \text{ if }I=J,\\
              \left(\frac{m(\hat{I})}{\mu(I)}\right)^{p-1}\frac{m(\hat{J})}{\mu(J)} , \text{ if }\hat{I}=(\hat{J})^s,\\
             \left(\frac{m(\hat{I})}{\mu(I)}\right)^{p-1}\frac{m(J)}{\mu(J)} , \text{ if }J=(\hat{I})^s,\\
              \left(\frac{m(I)}{\mu(I)}\right)^{p-1}\frac{m(\hat{J})}{\mu(J)} , \text{ if }I=(\hat{J})^s,\\
             0, \text{ for any other case}.
        \end{cases}
    \end{equation*}
    
\begin{prop}\label{sibweightsprop} Let $1<p<\infty$ and $i\in\{1,2,3\}$. For all $\omega\in A_p^{sib}(\mu)$, for all $\eta$ sparse families $\mathcal{S}$, 
\begin{equation}
|\mathcal{E}_i^{\mathcal{S}} (f,g\omega)|\lesssim_{p,\eta} [\mu]_{sib}^{\gamma(p)}[\omega]^{1/p}_{A_p^{sib}(\mu)}[\omega]_{A_p(\mu)}^{\frac{(p-1)^2+1}{p(p-1)}} \|f\|_{L^{p}(\omega)}\|g\|_{L^{p'}(\omega)},
\end{equation}    
where $\gamma(p)=1/2-\min\{1/p,1/p'\}$. 
\end{prop}

\begin{proof}
We only deal with $\mathcal{E}_1^{\mathcal{S}}$ because the other two forms are similar. Since $\mu$ is sibling balanced, then for all $1<p<\infty$
 $$\sqrt{m(I)}\sqrt{m(I^s)}\leq [\mu]_{sib}^{1/2-\min\{1/p,1/p'\}}m(I)^{1/p}m(I^s)^{1/p'}.$$
 Using the standard notation $\langle f \rangle_{I,\omega} = \frac{1}{\omega(I)}\int_I f \omega d\mu$, with $\omega(I)=\int_I \omega d\mu$
\begin{equation*}
    \begin{split}
\mathcal{E}^{\mathcal{S}}_1(f,g\omega)\leq & [\mu]_{sib}^{\gamma(p)}\sum_{\substack{I,J\in \mathcal{S}\\ \hat{I}=(\hat{J})^{s}}}\avfI \langle g\omega \rangle_{J} \,m(\hat{I})^{1/p}m(\hat{J})^{1/p'} \\
\leq &[\mu]_{sib}^{\gamma(p)}\sum_{\substack{I,J\in \mathcal{S}\\ \hat{I}=(\hat{J})^{s}}}\langle f\sigma^{-1}\rangle_{I,\sigma} \langle g\rangle_{J,\omega} \sigma(I)^{1/p}\omega(J)^{1/p'}\frac{\omega(J)^{1/p}\sigma(I)^{1/p'}}{\mu(J)\mu(I)}\,m(\hat{J})^{1/p'}m(\hat{I})^{1/p}\\
=&[\mu]_{sib}^{\gamma(p)}\sum_{\substack{I,J\in \mathcal{S}\\ \hat{I}=(\hat{J})^{s}}}\langle f\sigma^{-1}\rangle_{I,\sigma} \langle g\rangle_{J,\omega} \sigma(I)^{1/p}\omega(J)^{1/p'} \left(\frac{\omega(J)}{\mu(J)}\frac{\sigma(I)^{p-1}}{\mu(I)^{p-1}}c(J,I)\right)^{1/p}\\
\leq & [\mu]_{sib}^{\gamma(p)}[\omega]_{A_{sib}^p(\mu)}^{1/p}\sum_{\substack{I,J\in \mathcal{S}\\
\hat{I}=(\hat{J})^{s}}}\langle f\sigma^{-1}\rangle_{I,\sigma} \langle g\rangle_{J,\omega} \sigma(I)^{1/p}\omega(J)^{1/p'}.
    \end{split}
\end{equation*}
Since $\mathcal{S}$ is $\eta$-sparse, 
 \begin{equation*}
     \begin{split}
        \eta \omega(I)^{1/p} \leq \frac{\mu(E_I)}{\mu(I)}\omega(I)^{1/p}=&\frac{\omega(I)^{1/p}}{\mu(I)}\int_{E_I} \omega^{1/p}\omega^{-1/p}d\mu\\
         \leq &\frac{\omega(I)^{1/p}}{\mu(I)}\omega(E_I)^{1/p}\sigma(I)^{1/p'}\leq \omega(E_I)^{1/p} [\omega]_{A_p(\mu)}^{1/p},
     \end{split}
 \end{equation*}
 so $\omega(I)\lesssim_{\eta} [\omega]_{A_p(\mu)} \omega(E_I)$. Similarly, $\sigma(E_I)\lesssim_{\eta} [\sigma]_{A_{p'}(\mu)} \sigma(E_I)=[\omega]_{A_p(\mu)}^{\frac{1}{p-1}} \sigma(E_I) $. Hence, 
 \begin{equation*}
     \begin{split}
\mathcal{E}^{\mathcal{S}}_1(f,g\omega)\leq [\mu]_{sib}^{\gamma(p)}[\omega]_{A_{sib}^p(\mu)}^{1/p}[\omega]_{A_p(\mu)}^{\frac{(p-1)^2+1}{p(p-1)}}\sum_{\substack{I,J\in \mathcal{S}\\
\hat{I}=(\hat{J})^{s}}}\langle f\sigma^{-1}\rangle_{I,\sigma} \langle g\rangle_{J,\omega} \sigma(E_I)^{1/p}\omega(E_J)^{1/p'}.
     \end{split}
 \end{equation*}
 Finally, one has
 \begin{equation*}
 \begin{split}
         \sum_{\substack{I,J\in \mathcal{S}\\
\hat{I}=(\hat{J})^{s}}}\langle f\sigma^{-1}\rangle_{I,\sigma} \langle g\rangle_{J,\omega} &\sigma(E_I)^{1/p}\omega(E_J)^{1/p'}\lesssim  (\sum_{I\in \mathcal{S}} \langle f\sigma^{-1}\rangle_{I,\sigma}\sigma(E_I))^{1/p} (\sum_{J\in \mathcal{S}} \langle g\rangle_{J,\omega}^{p'}\omega(E_J))^{1/p'} \\
\leq & \|\MaximalD^{\sigma}(f\sigma^{-1})\|_{L^{p}(\sigma)} \|\MaximalD^{\omega}(g)\|_{L^{p}(\omega)}\lesssim_{p} \|f\|_{L^{p}(\omega)}\|g\|_{L^{p'}(\omega)}.
 \end{split}
 \end{equation*}
\end{proof}

\begin{cor} \label{weightedresultforH}
Let $1<p<\infty$. If $\omega \in A_p^{sib}(\mu)$, then $\Hd$ extends to a bounded operator on $L^p(\omega).$ In particular, $\Hd$ is bounded on $L^p(\mu)$ for all sibling balanced, atomless $\mu$.
\end{cor}
\begin{proof}
By, \cite[Theorem 3.1]{Moen},
$$|\mathcal{A}_{\mathcal{S}}(f,g\omega)|\lesssim_{\eta,p} [\omega]_{A_{p}(\mu)}^{\min\{1,p'/p\}}\|f\|_{L^{p}(\omega)}\|g\|_{L^{p'}(\omega)}.$$
Therefore, by duality, Theorem \ref{sparseresultforH} and Proposition \ref{sibweightsprop} we get that for all $\omega\in A_p^{sib}(\mu)$, 
\begin{equation*}
    \begin{split}
          \|\Hd \|_{L^{p}(\omega)\rightarrow L^{p}(\omega)}=&\sup_{\|f\|_{L^{p}(\omega)}=1,\|g\|_{L^{p'}(\omega)}=1}\left|\int\Hd fg\omega d\mu\right|\\
          \lesssim_p &[\mu]_{sib}^{1/2} [\omega]_{A_p(\mu)}^{\min\{1,p'/p\}}+[\mu]_{sib}^{\gamma(p)}[\omega]^{1/p}_{A_p^{sib}(\mu)}[\omega]_{A_p(\mu)}^{\frac{(p-1)^2+1}{p(p-1)}}.
    \end{split}
\end{equation*}
\end{proof}



\bibliographystyle{alpha}
\bibliography{references}


\end{document}